\documentclass[12pt, a4paper,reqno]{amsart}

\usepackage[inner=3cm, outer=3cm]{geometry}
\usepackage[english]{babel}
\usepackage[T1]{fontenc}
\usepackage{amsmath}
\usepackage{float}
\usepackage{mathtools}
\usepackage{amsfonts}
\usepackage{amssymb}
\usepackage{setspace}
\usepackage{amsfonts, dsfont}
\usepackage{mathrsfs, enumerate, csquotes, color}
\usepackage{xcolor}
\usepackage[bb=boondox]{mathalfa}
\definecolor{vertfonce}{rgb}{0.20, 0.46, 0.25}
\definecolor{rougefonce}{rgb}{0.64, 0.09, 0.20}
\usepackage[breaklinks=true,
colorlinks=true,
linkcolor=rougefonce,
citecolor=vertfonce]{hyperref}

\usepackage{tikz}
\usetikzlibrary{shapes.misc}
\tikzset{cross/.style={cross out}, minimum size=1pt, draw=black, inner sep =0pt, outer sep=0pt, cross/.default={1pt}}

\theoremstyle{plain}
\newtheorem{theorem}{Theorem}[section]
\newtheorem{lemma}[theorem]{Lemma}
\newtheorem{corollary}[theorem]{Corollary}
\newtheorem{proposition}[theorem]{Proposition}

\theoremstyle{definition}
\newtheorem{remark}[theorem]{Remark}

\newtheorem{assumption}[theorem]{Assumption}

\newcommand{\R}{\mathbb{R}}
\newcommand{\T}{\mathbb{T}}
\newcommand{\C}{\mathbb{C}}
\newcommand{\N}{\mathbb{N}}

\renewcommand{\Im}{\mathrm{Im}}

\newcommand{\cO}{\mathscr{O}}

\renewcommand{\leq}{\leqslant}
\renewcommand{\geq}{\geqslant}

\newcommand{\eps}{\varepsilon}

\def\restriction#1#2{\mathchoice
              {\setbox1\hbox{${\displaystyle #1}_{\scriptstyle #2}$}
              \restrictionaux{#1}{#2}}
              {\setbox1\hbox{${\textstyle #1}_{\scriptstyle #2}$}
              \restrictionaux{#1}{#2}}
              {\setbox1\hbox{${\scriptstyle #1}_{\scriptscriptstyle #2}$}
              \restrictionaux{#1}{#2}}
              {\setbox1\hbox{${\scriptscriptstyle #1}_{\scriptscriptstyle #2}$}
              \restrictionaux{#1}{#2}}}
\def\restrictionaux#1#2{{#1\,\smash{\vrule height .8\ht1 depth .85\dp1}}_{\,#2}}

\author{L. Benedetto}
\author{C. Fermanian Kammerer}
\author{N. Raymond}
\author{É. Vacelet}

\title[The superadiabatic projectors method]{The superadiabatic projectors method applied to the spectral theory of magnetic operators}

\makeatletter

\@addtoreset{equation}{section}
\makeatother

\begin{document}
	
\maketitle

\begin{abstract}
This article deals with a generalization of the superadiabatic projectors method. In a general framework, the well-known superadiabatic projectors are constructed and accurately described in the case of rank one, when a remarkable factorization occurs. We apply these ideas to spectral theory and we explain how our abstract results allow to recover or improve recent results about the semiclassical magnetic Laplacian.
\end{abstract}

	
\section{Introduction} \label{sec:intro}

\subsection{Motivation and context}

This article deals with a generalization of the superadiabatic projectors method. In essence, this method aims at diagonalizing pseudodifferential operators $H^w$ by finding approximate eigenspaces. To do so, under suitable assumptions, we build a microlocal projection $\Pi^w$ that commutes microlocally with~$H^w$. This leads to consider an effective operator, formally given by $\Pi^wH^w\Pi^w$, attached to the spectral window we are interested in. The projection is constructed by means of the spectral elements of the operator-valued symbol $H$ and their quantization. In the present article, we establish a generalized version of this method (to operator-valued symbols in a non-necessarily selfadjoint context), and we apply it to spectral theory. In particular, we recover (and even improve), by elementary means, recent results about the semiclassical magnetic Laplacian in two dimensions (see \cite{Morin_Raymond_VuNgoc,Fahs_LeTreust_Raymond}). It turns out that this method is an alternative to the famous Grushin method (see the review article~\cite{SZ07} and the concise presentation in \cite{Martinez07}), which has recently shown its power in the semiclassical spectral analysis of magnetic Laplacians (see the Ph.D. thesis~\cite{keraval} and, for instance, \cite{BHR22, AAHR23}). The strategy that we propose in the present article simplifies a little the Grushin method in the sense that the latter relates the bijectivity of $H^w-z$ to that of an effective operator $a^w_{\mathrm{eff},z}$ that implicitly depends on $z$, whereas the microlocal projection method directly reduces to an effective operator independent of $z$. Hopefully, this article will inspire a new view point on the spectral analysis of magnetic Laplacians and suggest further applications. 

\subsection{Superadiabatic projectors} \label{sec:mathana}  
We consider two Hilbert spaces $\mathscr{A}$ and $\mathscr{B}$, with continuous embedding $\mathscr{A}\subset\mathscr{B}$, and we denote by $h>0$ a small semiclassical parameter. We consider the symbol class
\begin{multline*}
    S(\R^{2d},\mathscr{L}(\mathscr{A},\mathscr{B}))=\{H\in C^\infty(\R^{2d},\mathscr{L}(\mathscr{A},\mathscr{B}))\,, \forall \alpha\in\mathbb{N}^{2d}\,, \exists C>0 : \\
    \|\partial^\alpha H\|_{\mathscr{L}(\mathscr{A},\mathscr{B})}\leq C_\alpha \}\,.  \end{multline*}
The operator-valued symbols can depend on $h$, as soon as the estimates defining the symbol class hold uniformly for $h$ small. Of course, contrary to the scalar case, this class of symbol is not an algebra. If $H\in  S(\R^{2d},\mathscr{L}(\mathscr{A},\mathscr{B}))$, the semiclassical Weyl quantization of $H$ is given by
\[
    H^w u(x) = \frac{1}{(2\pi h)^n}\iint_{\R^{2d}} e^{i\frac{(x-y)\cdot \xi}{h}}H\left(\frac{x+y}{2},\xi\right)u(y)\,\mathrm{d}y\mathrm{d}\xi\,,
\]
for $u\in\mathscr{S}(\R^d,\mathscr{A})$.
For details about pseudodifferential operators with operator-valued symbols, one refers to \cite[Chapitre 2]{keraval}. The Moyal product induced by the composition of such operators (when it is well-defined) is denoted by $\circledast$.

In what follows, we consider a symbol $H\in S(\R^{2d},\mathscr{L}(\mathscr{A},\mathscr{B}))$ which admits the following asymptotic expansion
\[
    H \sim \sum_{j\in\N} h^j H_j.
\]
We call admissible symbol any element of $S(\R^{2d},\mathscr{L}(\mathscr{A},\mathscr{B}))$ having a similar asymptotic expansion.

\begin{assumption}\label{eq.assumption}
We work in the following framework.
\begin{enumerate}[\rm (i)]
    \item For all $(x,\xi)\in\R^{2d}$, the principal symbol $H_0(x,\xi)$ is a selfadjoint operator on the Hilbert space $\mathscr{B}$ with domain $\mathscr{A}$ and we denote by $\sigma(x,\xi) \subset \R$ its spectrum.
    \item There exists $\delta>0$ such that the following holds: for all $(x,\xi)\in\R^{2d}$, there exists a subset $\sigma_0(x,\xi)\subset \sigma(x,\xi)$ of the spectrum of $H_0(x,\xi)$ such that
\begin{equation}
    \label{eq:spectrumisolated}
    {\rm dist}(\sigma_{0}(x,\xi),\sigma(x,\xi)\setminus \sigma_0(x,\xi)) > \delta,
\end{equation}
and such that the corresponding spectral projector
\begin{equation}
    \label{eq:definitionPi0}
    \Pi_0(x,\xi) \coloneqq \mathbf{1}_{\sigma_0(x,\xi)}\left(H_0(x,\xi)\right)
\end{equation}
defines a symbol in $S(\R^{2d},\mathscr{L}(\mathscr{B},\mathscr{A}))$.
    \item There exist $m,M\in\R$ such that, for all $(x,\xi)\in\R^{2d}$, $\sigma_0(x,\xi)\subset [m,M]$.
\end{enumerate}
\end{assumption}
Note that $\Pi_0(x,\xi)$ is bounded and thus define an operator of $\mathscr{L}(\mathscr{A})$. As such, it is an element of $\mathscr{L}(\mathscr{A}) \cap \mathscr{L}(\mathscr{B})$.
\smallskip

The first main result in this article is the following theorem.
\begin{theorem}\label{thm:main}
Let $H \in S(\R^{2d},\mathscr{L}(\mathscr{A},\mathscr{B}))$ be an admissible operator-valued symbol and $H^w$ its Weyl quantization. We assume that Assumption~\ref{eq.assumption} holds. Then, there exists an admissible operator-valued symbol $\Pi\in S(\R^{2d},\mathscr{L}(\mathscr{B},\mathscr{A}))$ with principal symbol $\Pi_0$ and unique up to a remainder of order $\mathscr{O}(h^\infty)$, satisfying
\[
\Pi^w\Pi^w=\Pi^w+\mathscr{O}(h^\infty),
\]
where the remainder $\mathscr{O}(h^\infty)$ is a pseudodifferential operator whose symbol is in the same class as $\Pi$, and 
\[
[H^w,\Pi^w]=\mathscr{O}(h^\infty)\,,
\]
where the remainder $\mathscr{O}(h^\infty)$ is a pseudodifferential operator whose symbol is in the class as $H$. Moreover, if $H$ is selfadjoint then $\Pi$ is selfadjoint as well.
\end{theorem}

Theorem \ref{thm:main} is part of a rather long story, which originates from the famous Born-Oppenheimer approximation for molecules \cite{BO27}. On the mathematical side, this approximation method is often named adiabatic theory. In the early eighties, the adiabatic reduction for systems of Schr\"odinger equations has been performed by Nenciu in the context of the time-dependent Born-Oppenheimer approximation, see \cite{N1,N2}. The reader might want to consider the monographs \cite{Teufel} and \cite{Martinez_Sordoni}, which give detailed accounts of more recent approaches. Many insights and fundamental ideas are also described in \cite{SpohnTeufel,Martinez_Sordoni_CRAS,Robert21}. The idea of constructing superadiabatic projectors, that is almost projectors as quantization $\Pi^w$ of matrix-valued symbols, originates from the seminal work \cite{emwe} and the detailed construction in \cite{bily}. This strategy to study elliptic systems has also been used in \cite{Bolte_Glaiser,Bolte_Keppeler} and even more recently in \cite{Capoferri_Vassiliev_1,Capoferri_Vassiliev_2} in the case of matrix-valued symbols to obtain spectral results. Theorem \ref{thm:main} can be seen as a generalization of the above mentioned results to a more general class of (non necessarily selfadjoint) operator-valued symbols. As we will see, our proof will inspire an elegant way to derive an effective operator -- the adiabatic Hamiltonian -- having remarkable applications motivated by spectral theory (especially of semiclassical magnetic Laplacians), see Section \ref{sec.intro-applications} below.
\smallskip 

When $\sigma_0(x,\xi)$ is made of disjoint parts, the corresponding superadiabatic projectors are orthogonal to each other.
\begin{proposition}
	\label{prop:orthogonalprojadiab}
	Let $\Pi^1$ and $\Pi^2$ be the symbols of superadiabatic projectors given by Theorem~\ref{thm:main} and corresponding to disjoint parts of the spectrum of $H_0$ that satisfy Assumption \ref{eq:spectrumisolated}, denoted respectively by 
	$\sigma_0^1(x,\xi)$ and $\sigma_0^2(x,\xi)$ for $(x,\xi)\in\R^{2d}$.
	Then, $(\Pi^1)^w$ and $(\Pi^2)^w$ are orthogonal up to $\mathscr{O}(h^\infty)$ in the sense that we have 
	\begin{equation*}
		(\Pi^1)^w (\Pi^2)^w = \mathscr{O}(h^\infty).
	\end{equation*}
\end{proposition}

\subsection{Towards spectral theory}\label{sec.intro-applications}
Let us now turn to the applications of Theorem~\ref{thm:main}.
When $\Pi_0$ is of rank one, we can describe the structure of the superadiabatic projector. More precisely, we will make the following assumption.
\begin{assumption}\label{assumption:existencebase}
There exists $(x,\xi)\in\R^{2d} \mapsto u_0(x,\xi)\in\mathscr{B}$ in $C^\infty(\R^{2d},\mathscr{B})$ such that 
	\begin{equation*}
		\forall (x,\xi)\in\R^{2d},\,\; \|u_0(x,\xi)\|_{\mathscr{B}}=1 \quad \mbox{and}\quad\Pi_0(x,\xi) = (u_0(x,\xi),\cdot)_{\mathscr{B}}\, u_0(x,\xi).
	\end{equation*}
\end{assumption}

Under this assumption, the set $\sigma_0(x,\xi)$, $(x,\xi)\in\R^{2d}$, is reduced to the point given by the eigenvalue corresponding to the eigenvector $u_0(x,\xi)$. Moreover, it allows for the following factorization of the superadiabatic projectors.

\begin{theorem}\label{thm:spectral}
	Under Assumptions~\ref{eq.assumption} and \ref{assumption:existencebase},
	there exist
	two symbols $\ell$ and $L$ in the class $S(\R^{2d},\mathscr{L}(\mathscr{A},\C))\cap S(\R^{2d},\mathscr{L}(\mathscr{B},\C))$
	such that 
	\begin{equation}\label{eq:ell_L}
		L^w (\ell^w)^* = {\rm Id}_{L^2(\R^d,\C)}+\mathscr{O}(h^\infty)\quad \mbox{and}\quad (\ell^w)^* L^w = \Pi^w+\mathscr{O}(h^\infty).
	\end{equation}
	Moreover, if $H$ is selfadjoint, then $L=\ell$. 
\end{theorem}

The existence of such operators as $L^w$ and $\ell^w$ allows us to diagonalize formally the operator $H^w$. All along the paper, we will use the notation $\Pi^\perp$ for denoting the operator $\Pi^\perp=1-\Pi$ for a projector
$\Pi$. In particular, if $\Pi$ is a projector $\Pi^\perp\Pi =\Pi\Pi^\perp=0$.

\begin{corollary}\label{cor:spectrum}
We continue with the notations of Theorem~\ref{thm:main}, and assume there exists~$\lambda\in\mathbb{R}$ and~$N\in\mathbb{N}$ such that, for all $(x,\xi)\in\mathbb{R}^{2d}$,
\[
\mathrm{sp}(H_0(x,\xi)\Pi_0(x,\xi)^\perp)\subset [\lambda,+\infty)
\]
and 
\[ 
\sigma_0(x,\xi)= \mathrm{sp}(H_0(x,\xi))\cap(-\infty,\lambda)=\sigma_1(x,\xi)\sqcup\ldots\sqcup\sigma_N(x,\xi)
\]
with the $\sigma_n$ satisfying Assumptions \ref{eq.assumption} and \ref{assumption:existencebase}.
Denote for all $j\in\{1,\ldots,N\}$, with the notations of Theorem~\ref{thm:spectral}, $(\Pi^j)^w=(\ell_j^w)^*L_j^w+\mathscr O(h^\infty)$
and set
\[M^w=\begin{pmatrix}
L_1^w H^w(\ell_1^w)^*&0&\cdots&0\\
0&L_2^w H^w(\ell_2^w)^*&&\vdots\\
0&0&\cdots&L_N^wH^w(\ell_N^w)^*
\end{pmatrix}\,.\]
Assume moreover that $H^w$ and $M^w$ have discrete spectra in $(-\infty,\lambda)$.

Then, for all $\mu\in\mathrm{sp}(H^w)\cap(-\infty,\lambda)$, if $\psi$ is a corresponding eigenfunction and letting $\Psi={}^t(L_1^w\psi,\ldots,L_N^w\psi)$, we have
\[M^w\Psi=\mu\Psi+\mathscr{O}(h^\infty)\|\Psi\|\,.\]
Conversely, for all $\mu\in\mathrm{sp}(M^w)\cap(-\infty,\lambda)$, if $\Phi={}^t(\Phi_1,\cdots,\Phi_N)$ is a corresponding eigenfunction and letting $\varphi=(\ell_1^w)^*\Phi_1+\ldots+(\ell_N^w)^*\Phi_N$, we have
\[H^w\varphi=\mu\varphi+\mathscr{O}(h^\infty)\|\varphi\|\,.\]
Moreover, in the selfadjoint case, the spectrum of $H^w$ and of $M^w$ in $(-\infty,\lambda)$ coincide modulo $\mathscr{O}(h^\infty)$\footnote{We refer to \cite[Definition 1.5 \& Remark 1.6]{Fahs_LeTreust_Raymond} for a precise definition of "coinciding modulo $\mathscr{O}(h^\infty)$", which takes the multiplicity into account.}.
\end{corollary}

Corollary \ref{cor:spectrum} and some of its adaptations can be useful in spectral theory, especially because the spectral analysis of $H^w$ is formally reduced to that of a family of scalar pseudodifferential operators, what is well-known in many situations (see, for instance, the classical references \cite{Rozenblum75, Sjostrand_1992}).
Indeed, by construction the $L_j^w H^w(\ell_j^w)^*$-s are scalar operators for all $j\in\{1,\cdots, N\}$.

\begin{remark}\label{rem.sub}
If $H_0$ does not depend on $x$, then $L_0$ and $\ell_0$ do not depend on $x$ and the subprincipal symbol of $L^w H^w\ell^w$ is 
\[
L_0H_1\ell^*_0+L_1 H_0\ell_0^*+L_0H_0\ell_1^*=L_0H_1\ell^*_0+\mu_0(\xi)(L_1 \ell_0^*+L_0\ell_1^*).
\]
Moreover, Theorem \ref{thm:spectral} gives in this case $L_1 \ell_0^*+L_0\ell_1^*=0$. Thus, 
\[
L^w H^w\ell^w=\mu_0(\xi)^w+h(L_0H_1\ell^*_0)^w+\mathscr{O}(h^2).
\]
\end{remark}

In the following section, we describe some applications/adaptations to the study of magnetic operators, which have seemingly not been noticed before.

\subsection{Application to electro-magnetic wells in two dimensions}
Let us revisit the article~\cite{Morin_Raymond_VuNgoc}. We consider the two-dimensional electro-magnetic Schrödinger operator 
\[
\mathscr L_h=(-ih\nabla -\mathbf{A}(q))^2 +hV(q),\;\; q=(q_1,q_2)\in\mathbb R^2,
\]
with $\mathbf{A}\in\mathcal C^\infty (\mathbb R^2,\mathbb R^2)$ and $V\in \mathcal C^\infty (\mathbb R, \mathbb R)$. We restrict ourselves to the study of the selfadjoint case (when $V$ is real-valued), and we use the same assumptions as in~\cite{Morin_Raymond_VuNgoc}.
\begin{assumption}\label{as:app_double_well} We work in the following framework.
\begin{enumerate}
    \item[(i)] The magnetic field $B(q):=\partial_{1}A_2(q)-\partial_{2}A_1(q)$ is non-vanishing and there exists $b_0>0$ such that 
    \[
    \forall q\in\mathbb R^2,\;\; B(q)\geq b_0.
    \]
    \item[(ii)] The functions $B$, $V$ and 
    \[
    q\mapsto \int_0^{q_1} \partial_2 B(s,q_2) ds,
    \]
     are all bounded scalar symbols, i.e in the set $S(\mathbb R^2,\mathbb R)$.
     \item[(iii)] The function $q\mapsto B(q)+V(q)$ admits a unique global minimum, not reached at infinity and we set 
     \[
     \mu_0:=\min (B+V).
     \]
\end{enumerate}
\end{assumption}

As a first application of the superadiabatic projectors, we get the following theorem, which is well-known when $V=0$ (see \cite{HK11,RVN15}).

\begin{theorem}\label{thm:app_magnetic_well}
Consider Assumption \ref{as:app_double_well}. Let $C>0$. There exist $h_0 > 0$ and $N\in\mathbb{N}$ such that for all $h \in (0,h_0]$, the spectrum of $\mathscr L_h$ in $\left( \mu_0h , \mu_0h + Ch^2 \right)$ is a family of discrete eigenvalues $\left( \lambda_j(h) \right)_{1 \leqslant j \leqslant N}$ satisfying
\[
\lambda_j(h) = \mu_0 h + \left( \left( 2k-1 \right) c_0 + c_1 \right)h^2 + h^{\frac{3}{2}} \sum_{k\geq 2} c_{j,k} h^{\frac{k}{2}} +\mathscr{O}(h^\infty)\,,
\]
with $(c_{j,k})_{k\in\mathbb N}\in{\mathbb R}^{\mathbb N}$, $c_1\in\R$, and $c_0>0$.
\end{theorem}

\subsection{Application to the Robin magnetic Laplacian}
Let us now revisit the article~\cite{Fahs_LeTreust_Raymond}, and in order to show the effectiveness of our approach, let us also improve some of its results. We consider
\[
    \mathscr L_h=(-ih\nabla -\mathbf{A})^2,
\]
on a smooth, bounded, simply connected open Euclidean domain $\Omega$ of $\mathbb R^2$ with boundary condition of Robin type:
\begin{multline*}
    {\rm Dom}(\mathscr L_h) = \{\psi\in H^1(\Omega) : (-ih\nabla - \mathbf{A})^2\psi \in  L^2(\Omega)\,,\\ -ih \mathbf{n} \cdot (-h\nabla- \mathbf{A})\psi  = \gamma h^{\frac 3
    2}\psi  \mbox{ on } \partial \Omega\}\,.
\end{multline*}
Note that the parameter $\gamma\in [0,+\infty]$ encodes the Robin condition: $\gamma= +\infty$ corresponds to Dirichlet conditions and $\gamma=0$ to Neumann ones. The vector potential~$\mathbf{A}$ is supposed to be smooth and generating a constant magnetic field of intensity $1$. We are interested in eigenvalues of the operator $\mathscr L_h$ in  windows of the form $[ha, hb]$ with $2n - 1< a < b < 2n + 1$ for some $n\in\mathbb N$.

\subsubsection{Dispersion curves}
We keep the notation of~\cite{Fahs_LeTreust_Raymond}. We denote by $\mathbb R_+$ the half line $\mathbb R_+:=\{t>0\}$ and we consider the  De Gennes operator
\[
    H[\gamma,\sigma] = -\frac{\mathrm d^2}{\mathrm d t^2} + (t-\sigma)^2,
\]
with domain 
\[
    \mathrm{Dom}\left( H[\gamma,\sigma] \right) = \left\{ u \in B^1 \left( \mathbb R_+ \right) :\; H[\gamma,\sigma]u \in L^2(\mathbb R_+), u'(0) = \gamma u(0) \right\},
\]
where $B^1 \left( \mathbb R_+ \right) = \left\{ u \in H^1 \left( \mathbb R_+ \right) :\; [t \mapsto tu(t)] \in L^2(\mathbb R_+) \right\}$. The operator $H[\gamma,\sigma]$ has compact resolvent. We denote by $\left(\mu_n(\gamma,\sigma)\right)_{n \geq 1}$ the non-decreasing sequence consisting in its  eigenvalues, and by $(u_n^{[\gamma,\sigma]})_{n\geq 1}$ the associated normalized eigenfunctions.

For all $n\in\mathbb N^*$, the behavior of the  functions $\sigma \mapsto \mu_n(\gamma,\sigma)$ is well understood. they have a unique non-degenerate minimum with critical value $\Theta^{[n-1]}(\gamma)$ in the interval $(2n-3,2n-1)$ (see \cite[Proposition 1.1]{Fahs_LeTreust_Raymond}). Moreover, if $a,b\in\mathbb N$ belong to a $h$-independent compact subset of the interval $(2n-3,2n-1)$, then for any $k\geq 1$, the set  $\mu^{-1}_k ([a, b])$ is a compact with a finite number of connected component. It is known (see \cite[Corollary 1.2]{Fahs_LeTreust_Raymond}) that 
\begin{equation}\label{def:N(a,b)}
    N(\gamma,a,b) := \sharp \{k\geq 1 : \mu_k(\gamma, \cdot)^{-1}([a, b])\not=\emptyset\} 
    =\left\{
    \begin{array}{cl}
    n     &\mbox{if}\; b\geq \Theta^{[n-1]}(\gamma),  \\
     n-1    & \mbox{otherwise}.
    \end{array}
    \right.
\end{equation}

\subsubsection{Statement}
For $L>0$, we set $\mathbb T_{2L}=\mathbb R/2L\mathbb Z$ and denote by $L^2(\mathbb T_{2L})$ the subspace of $L^2_{loc}(\mathbb R)$ consisting in the $2L$-periodic functions and  equipped with the usual $L^2$ norm on $[0, 2L]$. A slight adaptation of the analysis of the present article allows to prove the following theorem, which improves \cite[Corollary 1.10]{Fahs_LeTreust_Raymond} by giving a spectral description modulo $\mathscr{O}(h^\infty)$.

 \begin{theorem}\label{theo:Robin}
 Let $\gamma\in [0,+\infty]$ be the Robin parameter. 
Let $n\in\mathbb N$ and $a,b\in\mathbb N$ belonging to a $h$-independent compact subset of the interval  $(2n-3,2n-1)$. Let $N:=N(\gamma,a,b)$ given by~\eqref{def:N(a,b)}. We assume that, for all $k\in\{1,\ldots,N\}$, $\mu^{-1}_k(\gamma, [a,b])$ does not contain a critical point of $\mu_k(\gamma,\cdot)$.

Then, the spectrum of $\mathscr L_h$ in $[ha, hb]$ coincides with that of the operator 
$hD^w_h$ modulo $\mathscr O(h^\infty)$, where
\[
    D^w :=\begin{pmatrix}
    m_1^w  &  0    & \cdots & 0 \\
    0      & m_2^w & \cdots & 0 \\ 
    \vdots &       &       & \vdots \\
    0      &\cdots & 0    & m_N^w
    \end{pmatrix}
\]
is a bounded operator acting diagonally on ${\rm e}^{\frac i h \mathfrak f_0\cdot} L^2(\mathbb T_{2L})^N$. 
Here
\begin{equation}\label{def:f0}
\mathfrak f_0 = \frac{|\Omega|}{|\partial\Omega|}\;\;\mbox{and}\;\; 2L=|\partial\Omega|
\end{equation}
and each $m_k^w$ is a $\sqrt h$-pseudodifferential operator with symbol in $S(\mathbb R^2,\C)$ whose principal symbol depends on $\sigma$ only and coincides with $\mu_k(\gamma, \sigma)$ in a neighborhood of $\mu_k^{-1}([a,b])$.
\end{theorem}
\begin{remark}
The proof of Theorem \ref{theo:Robin}, combined with Remark \ref{rem.sub}, allows to recover the main result in \cite{Fahs_LeTreust_Raymond}, namely Theorem 1.7 which is a description of the spectrum modulo $\mathscr{O}(h^2)$ since the operators $m^w_j$ are selfadjoint modulo $\mathscr{O}(h^2)$.
\end{remark}

\subsection{Overview of the article} Section~\ref{sec:superadiabatic} is devoted to the proof of Theorem~\ref{thm:main}
(Section~\ref{sec:proof_main}), to the analysis of various properties of the superadiabatic projectors (Section~\ref{sec:properties}), to the proofs of Theorem~\ref{thm:spectral} and Corollary~\ref{cor:spectrum} (Section~\ref{sec:specandquasi}), and to a comparison of the superadiabatic projectors with the spectral ones when $H^w$ is  self-adjoint and bounded from below (Section~\ref{sec:sabb}). The two applications are detailed in Section~\ref{sec:application} and the Appendix~\ref{sec:sylvester} is devoted to the  Sylvester problem, which is a key argument in the proof of Theorem~\ref{thm:main}.

\subsection*{Acknowledgements} 
The authors thank Caroline Lasser and Fabricio Macià who indicated them the literature about the Sylvester problem, and for so many discussions concerning superadiabatic projectors.
The authors acknowledge the support  of  the Région Pays de la Loire via the Connect Talent Project HiFrAn 2022 07750, and  from the France 2030 program, Centre Henri Lebesgue ANR-11-LABX-0020-01.

\section{Superadiabatic projectors}\label{sec:superadiabatic}

In this section, we revisit the construction of superadiabatic projectors. Comparatively to existing results, we relax the assumptions on the subprincipal symbol of the original Hamiltonian since only $H_0$ and $\Pi_0$ are supposed to be selfadjoint. 

\subsection{Construction of  superadiabatic projectors}\label{sec:proof_main}

\subsubsection{Proof of Theorem~\ref{thm:main}}
One constructs, by induction, a symbol $\displaystyle{\Pi \sim \sum_{k\geq 0} h^k \Pi_k}$ satisfying the  relations
\begin{equation*}
\Pi \circledast \Pi = \Pi + \cO(h^\infty) \quad \mbox{and}\quad H \circledast \Pi - \Pi \circledast H = \cO(h^\infty),
\end{equation*}
where $\circledast$ denotes the Moyal product and the remainders $\cO(h^\infty)$ are taken in the sense of formal series in $h$ 
with coefficients in $C^\infty(\R^{2d},\mathscr{L}(\mathscr{A}))\cap C^\infty(\R^{2d},\mathscr{L}(\mathscr{B}))$. More precisely, our aim is to construct formally
the sequence of symbols $(\Pi_k)_{k\in\mathbb N}$ such that we have for all $k\in\N$
\begin{align}\label{eq:Pik}
	&\Pi^w_{[k]}\Pi^w_{[k]}=\Pi_{[k]}^w+h^{k+1}R_{k+1}^w\,,\\
	\label{eq:Hk}
	&[H^w_{[k]},\Pi^w_{[k]}]=h^{k+1}T^w_{k+1},
\end{align}
for some symbols $R_{k+1}=\sum_{j\geq 0}h^j R_{k+1,j}$ and $T_{k+1}=\sum_{j\geq 0}h^j T_{k+1,j}$, and where
\begin{align*}
&\Pi_{[k]} \coloneqq \Pi_0+h\Pi_1+\ldots+h^{k}\Pi_k\,,\\
&H_{[k]} \coloneqq H_0+hH_1+\cdots +h^kH_k.
\end{align*}

When $k=0$,  the symbol $\Pi_0$ readily satisfies the relations~\eqref{eq:Pik} and~\eqref{eq:Hk} by the composition theorem for pseudodifferential operators with operator-valued symbols. 
Let us now assume that we have constructed $\Pi_0,\Pi_1,\ldots,\Pi_k$ in such a way that we have~\eqref{eq:Pik} and~\eqref{eq:Hk}.
We are looking for $\Pi_{k+1}$ such that
\begin{equation}\label{eq:condition_k_0}
\Pi_{0}\Pi_{k+1}+\Pi_{k+1}\Pi_{0}=\Pi_{k+1}-R_{k+1,0}
\;\;\mbox{
and }\;\;
[H_0,\Pi_{k+1}]
+[H_{k+1},\Pi_0]
=-T_{k+1,0}
\,,
\end{equation}
or equivalently, satisfying
\begin{equation}\label{eq:condition_k}
	\Pi_{0}\Pi_{k+1}-\Pi_{k+1}\Pi_{0}^\perp=-R_{k+1,0} \;\;\mbox{
and }\;\;
[H_0,\Pi_{k+1}]=-S_{k+1,0}\,,
\end{equation}
where $S_{k+1,0}= T_{k+1,0}+ [H_{k+1},\Pi_0]$.

\smallskip 

 The first relation in~\eqref{eq:condition_k} determines the diagonal part of $\Pi_{k+1}$:
\begin{equation}\label{eq:Pi_k+1_diag}
\Pi_0 \Pi_{k+1}\Pi_0=- \Pi_0R_{k+1,0}\Pi_0\;\;\mbox{and}\;\;
\Pi_0^\perp \Pi_{k+1}\Pi_0^\perp= \Pi_0^\perp R_{k+1,0}\Pi_0^\perp.
\end{equation}
It has a solution if and only if one has the first compatibility relation 
\begin{equation}\label{comp_1}
[\Pi_0,R_{k+1,0}]=0.
\end{equation}
The second relation in~\eqref{eq:condition_k} determines the anti-diagonal part of $\Pi_{k+1}$. Indeed, it writes 
\begin{equation}\label{eq:rec_k}
[H_0,\Pi_0 \Pi_{k+1}\Pi_0^\perp]=-\Pi_0 S_{k+1,0}\Pi^\perp_0\,,\quad [H_0,\Pi^\perp_0 \Pi_{k+1}\Pi_0]=-\Pi^\perp_0 S_{k+1,0}\Pi_0\,.
\end{equation}
This system can be solved if and only if 
	\[\Pi_0[H_0,R_{k+1,0}]\Pi_0=\Pi_0 S_{k+1,0}\Pi_0\;\;\mbox{
	and}\;\;
	\Pi_0^\perp[H_0,R_{k+1,0}]\Pi_0^\perp=-\Pi_0^\perp S_{k+1,0}\Pi_0^\perp\,,\]
	or, equivalently 
\begin{equation}\label{comp_2}
\Pi_0[H_0,R_{k+1,0}]\Pi_0=\Pi_0 T_{k+1,0}\Pi_0\;\;\mbox{
	and}\;\;
	\Pi_0^\perp[H_0,R_{k+1,0}]\Pi_0^\perp=-\Pi_0^\perp T_{k+1,0}\Pi_0^\perp\,.
	\end{equation}
Postponing the verification of these conditions, we can reformulate the first equation in~\eqref{eq:rec_k}  
as follows: let  
$Y=-\Pi_0 S_{k+1,0}\Pi^\perp_0$,  we are looking for $X=\Pi_0\Pi_{k+1}\Pi_0^\perp$ such that 
\begin{equation}\label{sylv_Pi0}
(H_0\Pi_0) X - X(H_0\Pi_0^\perp) = Y.
\end{equation}
Such a  problem is solved by the  Sylvester's Theorem~\ref{thm:sylvester} (see~\cite{bhatia_rosenthal,Rosenblum_56} and Section~\ref{sec:sylvester} where a proof is given),  applied to $K_0=H_0\Pi_0$ and $K_1=H_0\Pi_0^\perp$. The assumptions of Theorem~\ref{thm:sylvester} are satisfied because of Assumption~\ref{eq.assumption}. 
The second equation in~\eqref{eq:rec_k} can be solved similarly: the operator $(\Pi_0^\perp \Pi_{k+1}\Pi_0)^*$ solves the Sylvester problem~\eqref{sylv_Pi0} with $Y=- (\Pi_0 S^*_{k+1,0}\Pi_0^\perp)$. 
\smallskip

It remains to verify that the compatibility relations~\eqref{comp_1} and~\eqref{comp_2} are satisfied.
Let us start with~\eqref{comp_1}.
We deduce from~\eqref{eq:Pik} that 
	\begin{align*}
		\Pi^w_{[k]}	R_{k+1}^w (1-\Pi^w_{[k]}	)&= h^{-k-1} \Pi^w_{[k]}	(\Pi^w_{[k]}\Pi^w_{[k]}-\Pi_{[k]}^w)(1-\Pi^w_{[k]}	)\\
		&= -h^{-k-1}(\Pi^w_{[k]}\Pi^w_{[k]}-\Pi_{[k]}^w)^2= -h^{k+1}R_{k+1}^wR_{k+1}^w.
	\end{align*}
	whence $\Pi_0R_{k+1,0} (1-\Pi_0)=0$. One argues similarly with $(1-\Pi_0)R_{k+1,0} \Pi_0$.
	
\noindent 	For proving~\eqref{comp_2},  we use~\eqref{eq:Hk} and~\eqref{eq:Pik} and write
	\begin{align*}
	\Pi^w_{[k]} T_{k+1}^w \Pi^w_{[k]}&=h^{-k-1} \Pi^w_{[k]}[H_{[k]}^w,\Pi^w_{[k]}]\Pi^w_{[k]} \\
	&=h^{-k-1}\left( \Pi^w_{[k]} H_{[k]}^w (\Pi^w_{[k]})^2 -( \Pi^w_{[k]})^2H_{[k]}^w  \Pi^w_{[k]}\right)\\
	&= \Pi^w_{[k]}H_{[k]}^w R_{k+1}^w-R_{k+1}^wH_{[k]} \Pi^w_{[k]}, 
	\end{align*}
	whence $\Pi_0H_0 R_{k+1,0} -R_{k+1,0}  H_0\Pi_0=\Pi_0 T_{k+1,0}\Pi_0$, which implies 
	\[
	\Pi_0[H_0, R_{k+1,0} ]\Pi_0=\Pi_0 T_{k+1,0}\Pi_0,
	\]
	since $[\Pi_0, R_{k+1,0}]=0$. 
	One argues similarly for the other relation, which concludes the induction argument. 

\smallskip

The smoothness of the symbols $\Pi_k$ that we have constructed, and the boundedness of their derivatives are consequences of the formula~\eqref{eq:Pi_k+1_diag} for the diagonal part of~$\Pi_{k}$, and of  Corollary~\ref{cor:smooth}~(1) for its off-diagonal part. 
Borel summation procedure, used as in \cite[Prop. 2.1.6]{keraval}, finishes the proof of Theorem~\ref{thm:main}.

\subsubsection{Remarks and comments}

The proof of Theorem~\ref{thm:main} calls for several comments and remarks. We first want to emphasize that the method extends to more general settings that one might find itself in applications, see Section~\ref{sec:application}.

\begin{remark} \label{rem:construct}
The construction of the superadiabatic projector does not depend on the special form of the Moyal product. It only uses that $H_0$ is selfadjoint and $\Pi_0$ is a projector commuting with $H_0$ for which Assumption \ref{eq.assumption} holds. As a consequence, the result
Theorem~\ref{thm:main} is also satisfied  by operators the symbols of which admit asymptotics expansions of the form  
\[
H = H_0+ \sum_{j\in\N} h^{\alpha_ j}  H_j
\]
where   $j\mapsto \alpha_j$ is a non decreasing sequence such that  $\mathbb N\subset \{\alpha_j, j\in\mathbb N\}$ and  such that for all $k,\ell\in\mathbb N$, $\alpha_k+\alpha_\ell\in  \{\alpha_j, j\in\mathbb N\}$. It is the case for example when $\alpha_j= j/{q_0}$ for some $q_0\in\N^*$.
\end{remark}
We also want to comment the boundedness assumptions of Theorem~\ref{thm:main}. 
The algebraic construction of the symbols $(\Pi_k)_{k\in\N}$  can be performed without assumptions on the boundedness of the operators $(H_k^w)_{k \in \N}$ and $(\Pi_k^w)_{k \in \N}$.

As observed in~\cite{Robert21,FLR_23} for example, the construction of superadiabatic projectors is still valid
when, for example, $(x,\xi)\in\R^{2d}\mapsto H(x,\xi)$ is of subquadratic growth, i.e. satisfies
\[
\forall \alpha\in\N^{2d},\;\;|\alpha|\geq 2,\;\; \exists C_\alpha,\;\;\|  \partial^\alpha
H_0\|_{\mathscr L(\mathscr A,\mathscr B)}\leq C_\alpha. 
\]
It may also happen that the adiabatic assumption, namely Assumption \ref{eq.assumption}, is only verified in an open subset $\Omega$ of the cotangent space. Then, the quantization of the symbols constructed via the  algebraic process could only be performed after microlocalization, as stated in the next lemma.

\begin{lemma}	\label{lem:localadiabatic}
	Let $\Omega\subset \R^{2d}$ where the adiabatic assumption (ii) in Assumption~\ref{eq.assumption} holds. Let $\Omega_1$ an open subset $\Omega$, strictly included in $\Omega$, and $\chi_1\in C_{c}^\infty(\Omega)$ such that $\chi_1=1$ on $\Omega_1$. There exist $R_1, R_2\in C^\infty(\R^{2d})$ such that $R_1=R_2=0$ on $\Omega_1$ and such that \begin{align*}
	  &  (\Pi\chi_1)^w (\Pi\chi_1)^w= (\Pi\chi_1)^w+R_1^w+\mathscr O(h^\infty)
	  \;\;\mbox{and}\;\;
	    [ H^w, (\Pi\chi_1)^w]=R_2^w+\mathscr O(h^\infty).
	\end{align*}
\end{lemma}

\begin{proof}
The result comes from the symbolic calculus, once one has noticed that for all $\alpha \in\N^{2d}$, one has $\partial^\alpha \chi_1=0$ on $\Omega_1$.
\end{proof}


\subsection{Properties of the superadiabatic projectors}\label{sec:properties}

This section is devoted to various properties satisfied by the superadiabatic projectors of Theorem \ref{thm:main}.

\subsubsection{Control of the derivatives}
We place ourselves in the setting of Lemma \ref{lem:localadiabatic} and we assume that $H$ depends smoothly of the parameters~$(x,\xi)$ in an open set $\Omega\subset\R^{2d}$ with uniformly bounded derivatives, i.e is in $S(\Omega,\mathscr L(\mathscr A,\mathscr B))$. Then, under Assumption~\ref{eq.assumption},  Theorem~\ref{thm:main} generates a  superadiabatic projector $\Pi=\sum_{j\geq 0} h^j\Pi_j$ which is a smooth function on $\Omega$. 

According to Remark~\ref{rem:deriv_proj},
Assumption~\ref{eq.assumption} implies bounds on the derivatives of the projector $\Pi_0$ in terms of the gap parameter~$\delta$: 
\begin{equation}\label{eq:der_proj}
\forall \alpha\in\N^{2d},\;\; \exists C_\alpha>0,\;\;\sup_{(x,\xi)\in\Omega}\| \partial^\alpha_{x,\xi} \Pi_0(x,\xi)\|_{\mathscr{L}(\mathscr B,\mathscr A)}\leq C_\alpha \delta^{-|\alpha|}.
\end{equation}
The next result describes the growth of the derivatives of the other elements $(\Pi_j)_{j\geq 1}$ as well.

\begin{lemma}\label{lem:derivative_proj}
Under the preceding assumptions, for all $j\in\N$ and  $\alpha\in\N^{2d}$, there exists $C_{\alpha,j}>0$ such that 
\[
\forall z\in\Omega,\;\; \| \partial_z^\alpha \Pi_j(z)\|_{\mathscr L(\mathscr B,\mathscr A)}\leq C_{\alpha,j} \delta^{-|\alpha|-2j}.
\]
\end{lemma}

\begin{proof}
We argue by induction and use Corollary~\ref{cor:smooth} (2). The result holds for $\Pi_0$. Let us now   assume that it holds for the $\Pi_j$-s with  $0\leq j\leq k$ for some $k\in\N$ and extend the result to $\Pi_{k+1}$. We revisit the construction of $\Pi_{k+1}$.
\smallskip 

The off-diagonal part  of the operator $\Pi_{k+1}$ is obtained by solving two Sylvester problems~\eqref{sylv_Pi0} with datas 
$\Pi_0 S_{k+1,0} \Pi_0^\perp$ and $-\Pi_0 S_{k+1,0}^*\Pi_0^\perp$
 respectively. Since $S_{k+1,0}=T_{k+1,0}-[H_{k+1},\Pi_0]$ where the $H_k$-s have uniformly bounded derivatives, the Equation~\eqref{eq:Hk} defining $T_{k+1,0}$, the formula of symbolic calculus in \cite[Théorème 2.1.12]{keraval} and the induction assumption imply
 \begin{align*}
\forall (x,\xi)\in\Omega,\;\;\|\ &
\partial^\alpha T_{k+1,0}(x,\xi)\|_{\mathscr L(\mathscr A,\mathscr B)}\leq c_{\alpha,k} \delta^{-|\alpha|-2k-1},
    \end{align*}
    for some $c_{\alpha,k}>0$. 
Corollary~\ref{cor:smooth} (2) then implies that there exists $\tilde c_{\alpha,k}>0$ such that  for all $(x,\xi) \in\Omega$, 
 \begin{align*}
&
\|\partial^\alpha (\Pi_0^\perp \Pi_{k+1} \Pi_0)(x,\xi))\|_{\mathscr L(\mathscr B,\mathscr A)}+
\|\partial^\alpha (\Pi_0 \Pi_{k+1} \Pi_0^\perp)(x,\xi))\|_{\mathscr L(\mathscr B,\mathscr A)}\leq \tilde c_{\alpha,k} \delta^{-|\alpha|-2k-2}.
    \end{align*}
Let us now consider the diagonal part of $\Pi_{k+1}$, which are given by Equation~\eqref{eq:Pi_k+1_diag}. Using again \cite[Théorème 2.1.12]{keraval}, the term $R_{k+1,0}$ involves terms of the form $\partial^{\beta_1}\Pi_{\ell_1} \partial^{\beta_2}\Pi_{\ell_2} $ with $\beta_1,\beta_2\in\N^{2d}$, $|\beta_1|=|\beta_2|=k-\ell_1-\ell_2$. Therefore, the induction equation implies that for $\alpha\in\N^{2d}$, there exists $C_{\alpha,\beta_1,\beta_2,\ell_1,\ell_2}>0$ such that 
\begin{align*}
    \|\partial^\alpha \left(
\partial^{\beta_1}\Pi_{\ell_1} \partial^{\beta_2}\Pi_{\ell_2} \right)\|_{\mathscr L(\mathscr B,\mathscr A)} &\leq C_{\alpha,\beta_1,\beta_2,\ell_1,\ell_2}  \delta ^{ -|\alpha| -|\beta_1|-2\ell_1-1 -|\beta_2|-2\ell_2-1}\\
&\leq C_{\alpha,\beta_1,\beta_2,\ell_1,\ell_2}\delta^{ -2(k+1) -|\alpha|}
\end{align*}
We deduce the existence of $\tilde c_{\alpha,k}>0$ such that 
\begin{align*}
\forall (x,\xi)\in\Omega,\;\;&\| 
\partial^\alpha R_{k+1,0}(x,\xi)\|_{\mathscr L(\mathscr A,\mathscr B)}\leq \tilde c_{\alpha,k} \delta^{-|\alpha|-2k-2}
    \end{align*}
and \eqref{eq:Pi_k+1_diag} gives the conclusion. 
\end{proof}

\subsubsection{Orthogonality: Proof of Proposition \ref{prop:orthogonalprojadiab}}
Suppose that we have constructed both symbols $\Pi^1$ and $\Pi^2$ of the superadiabatic projectors
as in the proof of Theorem~\ref{thm:main}. We set as before
\begin{equation*}
	\forall i\in\{1,2\},\,\forall k\in\N,\quad \Pi_{[k]}^i = \Pi_{0}^i + h\Pi_1^i +\cdots h^k \Pi_k^i\,.
\end{equation*}
We wish to prove that, for $k\in\N$, we have 
\begin{equation*}
	\Pi_{[k]}^1 \circledast \Pi_{[k]}^2 = \mathscr{O}(h^{k+1})\,,
\end{equation*}
which readily gives the desired result.
\smallskip

 The base case of the induction is true since 
$\Pi_0^1 \Pi_0^2 = 0$.
We now suppose that, for some $k\in\N$, there exists a symbol $X_{k+1}$ such that 
\begin{equation}
	\label{eq:defXk}
	\Pi_{[k]}^1\circledast \Pi_{[k]}^2 = h^{k+1}X_{k+1}\,.
\end{equation}
It is sufficient to establish that
\begin{equation}
	\label{eq:compatibilityXk}
 \Pi_0^1 (\Pi_{k+1}^2) + \Pi_{k+1}^1 (\Pi_0^2) +X_{k+1,0} =0.
\end{equation}

To prove this compatibility relation, we start by deriving other symbolic equations from the definition of $X_{k+1}$.
First, by Equation \eqref{eq:defXk}, we can write
\begin{align*}
	h^{k+1}\Pi_{[k]}^1\circledast X_{k+1}& = \Pi_{[k]}^1\circledast \Pi_{[k]}^1 \circledast \Pi_{[k]}^2\\& = \Pi_{[k]}^1 \circledast \Pi_{[k]}^2
	+ h^{k+1} R_{k+1}^1 \circledast \Pi_{[k]}^2\\ &= h^{k+1} (X_{k+1} +R^1_{k+1} \circledast \Pi_{[k]}^2),
\end{align*}
where $R_{k+1}$ is defined by Equation \eqref{eq:Pik} in the construction of the symbol $\Pi^1$. Doing similarly by composing on the right by
$\Pi_{[k]}^2$, we obtain the set of equations 
\begin{equation}
	\label{eq:compaproj}
	\begin{cases}
		&\Pi_0^{1,\perp}X_{k+1,0} + R_{k+1,0}^1\Pi_0^2 = 0,\\
		&X_{k+1,0}\Pi_0^{2,\perp}+ \Pi_0^1 R_{k+1,0}^2 = 0.
	\end{cases}
\end{equation}
Secondly, once again by definition of the symbol $X_k$, we can write 
\begin{equation*}
	\begin{aligned}
		h^{k+1} [H_{[k]},X_{k+1}]_{\circledast} &= [H_{[k]},\Pi_{[k]}^1\circledast \Pi_{[k]}^2]_{\circledast}\\
		&= [H_{[k]},\Pi_{[k]}^1]_{\circledast} \circledast \Pi_{[k]}^2+ \Pi_{[k]}^1\circledast [H_{[k]}, \Pi_{[k]}^2]_{\circledast}\\
		&= h^{k+1} T_{k+1}^1 \circledast \Pi_{[k]}^2+ h^{k+1} \Pi_{[k]}^1\circledast T_{k+1}^2.
	\end{aligned}
	\end{equation*}
Thus, we observe that the symbol $\Pi_0^1 X_{k+1,0}\Pi_0^2$ is solution of  
\begin{equation}\label{eq:dexieme_prop}
	[H_0,\Pi_0^1 X_{k+1,0}\Pi_0^2] = \Pi_0^1T_{k+1,0}^1\Pi_0^2+\Pi_0^1T_{k+1,0}^2\Pi_0^2.
\end{equation}
\smallskip 

Let us now prove~\eqref{eq:compatibilityXk}.
From~\eqref{eq:compaproj}, one checks readily that the desired compatibility relation \eqref{eq:compatibilityXk} composed on the left by
$\Pi_0^{1,\perp}$ and on the right by $\Pi_0^{2,\perp}$ holds. If we compose it on the left by $\Pi_0^{1,\perp}$ and on the right by $\Pi_0^2$,
we obtain 
\begin{equation*}
	\Pi_0^{1,\perp}X_{k+1,0}\Pi_0^2 + \Pi_0^{1,\perp}\Pi_{k+1}^1 \Pi_0^{1,\perp}\Pi_0^2 = 0.
\end{equation*}
This equation holds thanks to the fact that by Equation~\eqref{eq:Pi_k+1_diag}, we have 
\[
\Pi_0^{1,\perp}\Pi_{k+1}^1\Pi_0^{1,\perp}= \Pi_0^{1,\perp}R_{k+1,0}^1\Pi_0^{1,\perp}
\]
and using the first equation 
in \eqref{eq:compaproj}. We argue similarly  when composing Equation \eqref{eq:compatibilityXk} on the left by $\Pi_0^1$ and on the right by $\Pi_0^{2,\perp}$.

\smallskip

It remains  to check the compatibility relation~\eqref{eq:compatibilityXk} when composed on the left by $\Pi_0^1$ and on the right by $\Pi_0^2$, which writes 
\begin{equation}
	\label{eq:lastcompXk}
	\Pi_0^1X_{k+1,0}\Pi_0^2 + \Pi_0^1 \Pi_{k+1}^2 \Pi_0^2 + \Pi_0^1\Pi_{k+1}^1 \Pi_0^2 = 0.
\end{equation}
Recall that by~\eqref{eq:condition_k_0}, the symbols $\Pi_{k+1}^1$ and $\Pi_{k+1}^2$ satisfy the equations 
\begin{equation*}
	\begin{cases}
		&[H_0,\Pi_{k+1}^1]=- T_{k+1,0}^1- [H_{k+1},\Pi_0^1],\\
		&[H_0,\Pi_{k+1}^2]= -T_{k+1,0}^2- [H_{k+1},\Pi_0^2],
	\end{cases}
\end{equation*}
 with the symbols $T_{k+1}^i$, $i=1,2$, defined by Equation \eqref{eq:Hk}. Thus, since $\Pi_0^1$ and $\Pi_0^2$ commutes with $H_0$,
the symbol 
\[
P_{k+1} = \Pi_0^1 \Pi_{k+1}^2 \Pi_0^2 + \Pi_0^1\Pi_{k+1}^1 \Pi_0^2
\]
is solution of 
\begin{align*}
	[H_0,P_{k+1}] &=\Pi_0^1 [H_0, \Pi_{k+1}^2+\Pi_{k+1}^1 ]\Pi_0^2=- \Pi_0^1( T^1_{k+1} + T^2_{k+1})\Pi_0^2
\end{align*}
where we have used~\eqref{eq:condition_k_0}. Comparing with~\eqref{eq:dexieme_prop}, we deduce from 
the uniqueness of the solution of the Sylvester problem that $P_{k+1}=-\Pi_0^1 X_{k+1,0}\Pi_0^2$, which is 
the last equality  required for the compatibility relation~\eqref{eq:lastcompXk} to hold. This concludes
the proof of Proposition \ref{prop:orthogonalprojadiab}.

\subsection{Reduction to scalar operators and mutual quasimodes}
\label{sec:specandquasi}

Let us now prove Theorem~\ref{thm:spectral}. 

\subsubsection{A useful lemma}
We first construct the operators $\ell^w$ and~$L^w$ thanks to 
the following lemma.

\begin{lemma}\label{lem:existenceell}
	Under Assumption \ref{assumption:existencebase}, there exist two symbols $\ell$ and $L$ belonging to $S(\R^{2d},\mathscr{L}(\mathscr{A},\C))\cap S(\R^{2d},\mathscr{L}(\mathscr{B},\C))$ such that
	\begin{equation*}
		L^w (\ell^w)^* = {\rm Id}_{L^2(\R^d,\mathscr B)}+\mathscr{O}(h^\infty),\quad L^w = L^w \Pi^w+\mathscr{O}(h^\infty)\quad \mbox{and}\quad \ell^w = \ell^w \left( \Pi^w \right)^* + \mathscr{O}(h^\infty).
	\end{equation*}
\end{lemma}

\begin{proof}
	We construct, by induction, formal symbols $\ell, L$ of the form
	\begin{equation*}
		\ell = \sum_{k\geq 0} h^k \ell_k,
	\end{equation*}
	which satisfy the relations 
	\begin{equation*}
		L \circledast \ell^* = 1 +\mathscr{O}(h^\infty)\quad , \quad L = L \circledast \Pi+\mathscr{O}(h^\infty)\quad \mbox{and}\quad \ell = \ell \circledast \Pi^* + \mathscr{O}(h^\infty).
	\end{equation*}
	To this end, we introduce, as before, the partial symbols 
	\[
	\ell_{[k]} = \sum_{j=0}^k h^j \ell_j\;\; \mbox{and} \;\;L_{[k]} = \sum_{j=0}^k h^j L_j.
	\]
	
	\medskip
	
	\noindent 	By Assumption \ref{assumption:existencebase}, the base case of our induction is given by 
	\begin{equation*}
		\ell_0 = (u_0,\cdot)_\mathscr{B}\quad \mbox{and} \quad L_0 = (u_0,\cdot)_\mathscr{B}.
	\end{equation*} 
Now, suppose that for $k\in\N$, we have 
	\begin{align}\label{eq:remainderVW}
		&L_{[k]} \circledast \ell_{[k]}^* = 1 +h^{k+1} W_{k+1}, \\
		\nonumber 
		&L_{[k]} = L_{[k]}\circledast \Pi_{[k]} + h^{k+1}U_{k+1}, \\
		\nonumber
		&\ell_{[k]} = \ell_{[k]} \circledast \Pi_{[k]}^* + h^{k+1}V_{k+1}.
	\end{align}
	The conditions on $\ell_{k+1}$ and $L_{k+1}$ to obtain the next order in $h$ are given by
	\begin{equation*}
		\begin{aligned}
			&L_{k+1} \ell_0^* + L_0\ell_{k+1}^* =- W_{k+1,0},\\
			&L_{k+1} = L_{k+1}\Pi_0 + L_0 \Pi_{k+1} - U_{k+1,0},\\
			&\ell_{k+1} = \ell_{k+1} \Pi_0^* +  \ell_0 \Pi_{k+1}^* - V_{k+1,0}.
		\end{aligned}
	\end{equation*}
	As, for $(x,\xi)\in\R^{2d}$, $\ell_{k+1}(x,\xi)$ and $L_{k+1}(x,\xi)$ are linear forms on $\mathscr{B}$, we look for  $(x,\xi)\mapsto u_{k+1}(x,\xi) \in C^\infty(\R^{2d}, \mathscr{B})$ and $(x,\xi)\mapsto v_{k+1}(x,\xi) \in C^\infty(\R^{2d},  \mathscr{B})$ such that 
	$L_{k+1} = (u_{k+1},\cdot)_{\mathscr{B}}$ and $\ell_{k+1} = (v_{k+1},\cdot)_{\mathscr{B}}$. Then, the previous conditions become
	\begin{equation*}
		\begin{aligned}
			&(u_{k+1} , u_0)_{\mathscr{B}} + (u_0 , v_{k+1})_{\mathscr{B}} =- W_{k+1,0},\\
			&(u_{k+1},\cdot)_{\mathscr{B}} = (u_{k+1},\Pi_0 \cdot)_{\mathscr{B}} + (u_0,\Pi_{k+1}\cdot)_\mathscr{B}  -U_{k+1,0}(\cdot),\\
			&(v_{k+1},\cdot)_{\mathscr{B}} = (v_{k+1}, \Pi_0 \cdot)_{\mathscr{B}} + (u_0,\Pi_{k+1}^* \cdot)_\mathscr{B}  -V_{k+1,0}(\cdot).
		\end{aligned}
	\end{equation*}
	The first equation determines a relation between the coordinates of $u_{k+1}$ on ${\rm Ran}(\Pi_0)$ and the coordinates of $v_{k+1}$ on ${\rm Ran}(\Pi_0)$ while the second and the third determine $u_{k+1}$ and $v_{k+1}$ on ${\rm Ran}(\Pi_0^\perp)$.
	However, two compatibility conditions have to be fulfilled in order that  the two last equations have solutions: for all  $\varphi \in {\rm Ran}(\Pi_0)$, one  must have 
	\begin{equation}\label{comp_18}
	\begin{aligned}
		&(u_0,\Pi_{k+1}\varphi)_{\mathscr{B}} - U_{k+1,0}(\varphi)=0 \;\;\mbox{and}\;\;
		(u_0,\Pi_{k+1}^*\varphi) _{\mathscr{B}}- V_{k+1,0}(\varphi)=0.
	\end{aligned}
	\end{equation}
	This is indeed satisfied as one observes from Equation \eqref{eq:remainderVW} that we have 
	\begin{equation*}
		\begin{aligned}
			L_{[k]}\circledast \Pi_{[k]} &= L_{[k]} \circledast \Pi_{[k]} \circledast \Pi_{[k]} + h^{k+1} U_{k+1} \circledast \Pi_{[k]} \\
			&= L_{[k]} \circledast \Pi_{[k]} + h^{k+1} L_{[k]} \circledast R_{k+1}+ h^{k+1} U_{k+1} \circledast \Pi_{[k]}, \\
			\ell_{[k]} \circledast \Pi_{[k]}^* &= \ell_{[k]} \circledast \Pi_{[k]}^* \circledast \Pi_{[k]}^* + h^{k+1} V_{k+1} \circledast \Pi_{[k]}^* \\
			&= \ell_{[k]} \circledast \Pi_{[k]}^* + h^{k+1} \ell_{[k]} \circledast R_{k+1}^* + h^{k+1} V_{k+1} \circledast \Pi_{[k]}^*,
		\end{aligned}
	\end{equation*}
	where we have used Equation~\eqref{eq:Pik}. 
We deduce 	
	\begin{equation*}
	    \begin{aligned}
		&L_0 R_{k+1,0} + U_{k+1,0}\Pi_0 = 0\;\;\mbox{and}\;\;\ell_0 R_{k+1,0}^* + V_{k+1,0}\Pi_0 = 0.
		\end{aligned}
	\end{equation*}
	Since we have found that $\Pi_0 \Pi_{k+1}\Pi_0 = -\Pi_0 R_{k+1,0}\Pi_0$ in the proof of Theorem \ref{thm:main}, we obtain that for $\varphi\in{\rm Ran}(\Pi_0)$,
	\begin{align*}
	(u_0,\Pi_{k+1}\varphi)_{\mathscr{B}} -
	U_{k+1,0}(\varphi)& =(u_0,\Pi_{k+1}\varphi) _{\mathscr{B}}
	+L_0 R_{k+1,0} (\varphi)\\
	&= (u_0,\Pi_0\Pi_{k+1}\Pi_0\varphi)_{\mathscr{B}} +(u_0,R_{k+1,0} \varphi)_{\mathscr B}\\
	&=(u_0,\Pi_{k+1}\varphi)_{\mathscr{B}} - (u_0, \Pi_0\Pi_{k+1}\Pi_0 \varphi )_{\mathscr B}=0,
	\end{align*}
and	the first compatibility relation in~\eqref{comp_18} is satisfied. The other relation is derived  similarly, which concludes the proof. 
\end{proof}

\subsubsection{Proof of Theorem~\ref{thm:spectral}}
We start by proving the second relation of Equation~\eqref{eq:ell_L}.
	From Lemma \ref{lem:existenceell}, 
	we have
	\begin{equation*}
		L^w (\ell^w)^* = {\rm Id}_{L^2(\R^d,\mathscr B)}+\mathscr{O}(h^\infty).
	\end{equation*}
	We denote by $P^w = (\ell^w)^*L^w$. Our goal is to prove we have $P^w = \Pi^w + \mathscr{O}(h^\infty)$.
	
	First observe that up to $\mathscr{O}(h^\infty)$, $P^w$ is a projection:
	\begin{equation*}
		P^w P^w = (\ell^w)^*L^w (\ell^w)^*L^w = (\ell^w)^* \Big( {\rm Id}_{L^2(\R^d,\mathscr B)} + \mathscr{O}(h^\infty) \Big) L^w = P^w + \mathscr{O}(h^\infty).
	\end{equation*}
	Now, by  Lemma \ref{lem:existenceell}, we also have 
	\[
	L^w = L^w \Pi^w+\mathscr{O}(h^\infty),
	\]
	from which we deduce 
	\begin{equation} \label{eq:ppi}
		P^w\Pi^w = (\ell^w)^*L^w\Pi^w = (\ell^w)^*L^w + \mathscr{O}(h^\infty) = P^w + \mathscr{O}(h^\infty).
	\end{equation}
	Using that Lemma~\ref{lem:existenceell} yields $(\ell^w)^* = \Pi^w(\ell^w)^* + \mathscr{O}(h^\infty)$, we derive from a similar argument 
	\begin{equation} \label{eq:pip}
		\Pi^wP^w = \Pi^w(\ell^w)^*L^w = (\ell^w)^*L^w + \mathscr{O}(h^\infty) = P^w + \mathscr{O}(h^\infty).
	\end{equation}
	Now, by comparing the principal symbols of $P^w$ and $\Pi^w$,  we also have 
	\begin{equation*}
		P^w = \Pi^w + \mathscr{O}(h).
	\end{equation*}
	For concluding the proof, we consider the operator $R^w  \coloneqq \Pi^w - P^w$. Using
	\[
	\Pi^w\Pi^w=\Pi^w+\mathscr{O}(h^\infty),
	\;\;\mbox{and}\;\;
	P^wP^w=P^w+\mathscr{O}(h^\infty),
	\]
	and the  Equations~\eqref{eq:ppi} and~\eqref{eq:pip}, we deduce  that $R^w$ is a projection up to $\mathscr{O}(h^\infty)$, i.e.  $R^wR^w=R^w+\mathscr{O}(h^\infty)$. Therefore, 
	\[
	R^w \Big( {\rm Id}_{L^2(\R^d,\mathscr{B})} - \underbrace{R^w}_{=\mathscr{O}(h)} \Big) = \mathscr{O}(h^\infty).
	\]
	Since for $h$ small enough, ${\rm Id}_{L^2(\R^d, \mathscr{B}) }- R^w$ is bijective,  we conclude that $\Pi^w - P^w = \mathscr{O}(h^\infty)$, which terminates the proof.

\subsubsection{Proof of Corollary~\ref{cor:spectrum}}
We now focus on the proof of Corollary~\ref{cor:spectrum}.
\smallskip 

Let $\mu\in\mathrm{sp}(H^w)\cap(-\infty,\lambda)$ and $\psi$ a corresponding eigenfunction. We have $H^w\psi=\mu\psi$ so that, by using Theorem \ref{thm:main}, $H^w\Pi_j^w\psi=\mu\Pi_j^w\psi+\mathscr{O}(h^\infty)\|\psi\|$ for all $j\in\{1,\ldots,N\}$. Thanks to Theorem \ref{thm:spectral}, we deduce 
\[
H^w(\ell_j^w)^*L_j^w\psi=\mu (\ell_j^w)^*L_j^w\psi+\mathscr{O}(h^\infty)\|\psi\|
\]
and then 
\[
L_j^wH^w(\ell_j^w)^*L_j^w\psi=\mu L_j^w\psi+\mathscr{O}(h^\infty)\|\psi\|.
\] 
Setting $\Psi={}^t(L_1^w\psi,\ldots,L_N^w\psi)$, this shows that $M^w\Psi=\mu\Psi+\mathscr{O}(h^\infty)\|\psi\|$. Therefore, it remains to show that $\|\psi\|\leq C\|\Psi\|$. To see this, we observe  
\[
\Pi^w=\Pi_1^w+\ldots+\Pi_N^w+\mathscr{O}(h^\infty)
\]
(as a consequence of Theorem \ref{thm:main} and Proposition \ref{prop:orthogonalprojadiab}). This implies 
\[\|\Pi^w\psi\|^2\leq \sum_{j=1}^N\|L_j\psi\|^2+Ch\|\psi\|^2\,.\]
We have
\[(H^w-\mu)(1-\Pi^w)\psi+\Pi^w\psi=\Pi^w\psi+\mathscr{O}(h^\infty)\|\psi\|\,.\]
Then, we notice that $(H_0-\mu)\Pi_0^\perp+\Pi_0 : \mathscr{A}\to\mathscr{B}$ is bijective, with uniformly bounded inverse. Hence, with the composition theorem and the Calder\'on-Vaillancourt theorem, we get
\[\|\psi\|\leq C\|\Pi^w\psi\|\,.\]
This shows that $\|\psi\|\leq C\|\Psi\|$.
\smallskip 

Conversely, let $\mu \in \mathrm{sp}(M^w) \cap (-\infty,\lambda)$ and $\Phi$ a corresponding  eigenfunction. We have $L_j^w H^w(\ell_j^w)^*\Phi_j=\mu\Phi_j$ so that $(\ell_j^w)^*L_j^w H^w(\ell_j^w)^*\Phi_j=\mu(\ell_j^w)^*\Phi_j$. Thus, 
\[
H^w(\ell_j^w)^*L_j^w(\ell_j^w)^*\Phi_j=\mu(\ell_j^w)^*\Phi_j+\mathscr{O}(h^\infty)\|\Phi_j\|
\]
and 
\[
H^w(\ell_j^w)^*\Phi_j=\mu(\ell_j^w)^*\Phi_j+\mathscr{O}(h^\infty)\|\Phi_j\|.
\]
We let $\varphi = (\ell_1^w)^*\Phi_1+\ldots+(\ell_N^w)^*\Phi_N$ and we get $ H^w\varphi=\mu\varphi+\mathscr{O}(h^\infty)\|\Phi\|$. Thanks to Proposition \ref{prop:orthogonalprojadiab}, we notice that, for $m\neq n$, $\Pi_m^w\Pi_n^w=\mathscr{O}(h^\infty)$ so that 
\[
(\ell_m^w)^*L_m^w(\ell_n^w)^*L_n^w=\mathscr{O}(h^\infty).
\]
Hence, $L_m^w(\ell_n^w)^*=\mathscr{O}(h^\infty)$. This shows 
\[
\forall j\in\{1,\ldots,N\}, \;\;L_j^w\varphi=\Phi_j+\mathscr{O}(h^\infty)\|\Phi\|,
\]
which implies  $\|\Phi\|\leq C\|\varphi\|$.
\smallskip

When $H^w$ is selfadjoint, so is $M^w$. Therefore, in this case, the fact that the spectra of $M^w$ and $H^w$ coincide modulo $\mathscr{O}(h^\infty)$ is a consequence of the spectral theorem.

\subsection{Superadiabatic  and spectral projectors in the self-adjoint case}\label{sec:sabb}
We assume here that the operator $H^w$ under study is selfadjoint and revisit what bring the superadiabatic projectors of Theorem~\ref{thm:main} to the spectral theory of the operator $H^w$. The following proposition gives an insight into the relations between the superadiabatic projector $\Pi^w$ and the spectral functions of the Hamiltonian $H^w$. 

\begin{proposition}\label{prop:commutespectral}
Consider the same assumptions as in Theorem~\ref{thm:main} and assume that the operator $H^w$  is selfadjoint. Then, we have the following properties:
\begin{enumerate}
    \item  For any function $\chi\in C_{c}^\infty(\R)$ we have 
   $     [\chi(H^w),\Pi^w] = \mathscr{O}(h^\infty)$, in $L^2(\R^d,\mathscr{B})$.
    \item Assume that there exists an energy level $\lambda\in\R$ such that 
\begin{equation}\label{eq:localspectralH0}
		\exists \eps>0,\ \sigma(\restriction{H_{0}}{{\rm Ran}(\Pi_0^\perp)}) \subset (\lambda+\eps,+\infty).
\end{equation}
Then for any function $\chi\in C_{c}^\infty(\R)$ such that ${\rm supp}(\chi) \subset (-\infty,\lambda]$, we have in  $L^2(\R^d,\mathscr{B})$.
	\begin{equation*}
		 \Pi^w\chi(H^w) = \chi(H^w) + \mathscr{O}(h^\infty).
	\end{equation*}
\end{enumerate}
\end{proposition}

Point (2) of Proposition~\ref{prop:commutespectral} shows that, when $H^w$ is bounded from below, the range of $\Pi^w$ contains (up to $\mathscr{O}(h^\infty)$) the range of the spectral projector $H^w\mathbb{1}_{(-\infty,\lambda)}(H^w)$. 



The proof of Proposition~\ref{prop:commutespectral}  relies on the Helffer-Sjöstrand formula.      Let $\chi \in C_{c}^\infty(\R)$, as $H^w$ is a selfadjoint operator, we can write
    \begin{equation}\label{eq:HelfferSjostrand}
        \chi(H^w) = \frac{1}{\pi i}\int_{\C} \bar{\partial} \tilde{\chi}(z) \left(H^w-z\right)^{-1}L(\mathrm{d}z),
    \end{equation}
    where $L(\mathrm{d}z) = \mathrm{d}x\mathrm{d}y$ denotes the Lebesgue measure and $\tilde{\chi}\in C_{c}^\infty(\C)$ denotes an almost analytic extension of $\chi$ to the complex plane, i.e. satisfying
    \begin{equation*}
        \tilde{\chi}_{|\R} = \chi\quad\mbox{and}\quad \bar{\partial} \tilde{\chi}_{|\mathbb R} = 0,\quad\mbox{where}\quad \bar{\partial} = \frac{1}{2}(\partial_x + i\partial_y).
    \end{equation*}
    Such almost analytic extensions with compact support are constructed in \cite{Zwobook}; in addition, one can ask for the extension to satisfy
    \begin{equation}\label{control_tildechi}
        \bar{\partial}\tilde{\chi}(z) = \mathscr{O}(|\Im z|^3).
    \end{equation}
    For such a choice of extension, the Helffer-Sjöstrand formula holds.

\begin{proof}[Proof of Proposition \ref{prop:commutespectral}]
(1) By construction in Theorem~\ref{thm:main}, the 
  superadiabatic projector $\Pi^w$, satisfies for any $z\in\C\setminus\R$,
    \begin{equation*}
        [\Pi^w,H^w-z] = \mathscr{R}^w,
    \end{equation*}
    where $\mathscr{R}$ is of order $\mathscr{O}(h^\infty)$ in the symbol class $S(\R^{2d},\mathscr{L}(\mathscr{A},\mathscr{B}))$. We deduce from this that we have
    \begin{equation*}
        [\left(H^w-z\right)^{-1},\Pi^w] = \left(H^w-z\right)^{-1}\mathscr{R}^w \left(H^w-z\right)^{-1},
    \end{equation*}
    which gives immediately
    \begin{equation*}
         [\left(H^w-z\right)^{-1},\Pi^w] = \mathscr{O}(|\Im z|^{-2} h^\infty),
    \end{equation*}
    where the remainder is understood as a bounded operator on $L^2(\R^d,\mathscr{B})$.
    Then, using formula~\eqref{eq:HelfferSjostrand}, since the almost analytic extension $\tilde{\chi}$ is of compact support, we obtain the estimate of Proposition \ref{prop:commutespectral}, i.e
    \begin{equation*}
        \Pi^w \chi(H^w) = \chi(H^w) + \mathscr{O}(h^\infty).
    \end{equation*}

(2)    We now make the assumption that the principal symbol~$\Pi_0$ of the superadiabatic projector~$\Pi^w$ satisfies Equation \eqref{eq:localspectralH0} for some energy level $\lambda\in\R$ and some small parameter $\eps>0$.
    Then, for $\chi\in C_{c}^\infty(\R)$ with support included in the interval $(-\infty,\lambda)$, we have 
    \begin{equation}\label{eq:HScomposition}
        \chi(H^w) \Pi^{w,\perp}= \frac{1}{\pi i}\int_{\C} \bar{\partial} \tilde{\chi}(z) \left(H^w-z\right)^{-1}\Pi^{w,\perp}\,L(\mathrm{d}z)
    \end{equation}
  where $\Pi^{w,\perp}=1-\Pi^w$. We wish to prove the following equality:
    \begin{equation}\label{eq:egaliteinverse}
        \left(H^w-z\right)^{-1}\Pi^{w,\perp} = \left(\Pi^{w,\perp}H^w \Pi^{w,\perp}-z\right)^{-1}\Pi^{w,\perp} + \mathscr{O}(|\Im z|^{-2} h^\infty),
    \end{equation}
    where the remainder is understood as an operator on $L^2(\R^d,\mathscr{B})$. 
\smallskip

Let us prove~\eqref{eq:egaliteinverse}.
    Let $\psi \in L^2(\R^d,\mathscr{B})$ and let us compare
    \begin{equation*}
        \phi_1 = \left(H^w-z\right)^{-1}\Pi^{w,\perp}\psi\quad\mbox{and}\quad \phi_2 = \left(\Pi^{w,\perp}H^w \Pi^{w,\perp}-z\right)^{-1}\Pi^{w,\perp}\psi.
    \end{equation*}
     By construction, $\phi_1,\phi_2 \in L^2(\R^d,\mathscr{A})$   and   satisfy
    \begin{equation*}
        \|\phi_1\|_{L^2(\R^d,\mathscr{A})},\|\phi_2\|_{L^2(\R^d,\mathscr{A})} = \mathscr{O}(|\Im z|^{-1} \|\psi\|_{L^2(\R^d,\mathscr{B})}).
    \end{equation*}
    First, we observe 
    \begin{equation*}
        \Pi^w (H^w-z)\phi_1 = \Pi^w \Pi^{w,\perp}\psi = \mathscr{O}(h^\infty\|\psi\|_{L^2(\R^d,\mathscr{B})}).
    \end{equation*}
    Since we have the estimate $[\Pi^w,H^w-z] = \mathscr{O}(h^\infty)$, with the remainder understood as a bounded operator from $L^2(\R^d,\mathscr{A})$ to $L^2(\R^d,\mathscr{B})$, we obtain
    \begin{equation*}
        (H^w-z)\Pi^w \phi_1 = \mathscr{O}(h^\infty\|\psi\|_{L^2(\R^d,\mathscr{B})}) + \mathscr{O}(h^\infty\|\phi_1\|_{L^2(\R^d,\mathscr{A})}),
    \end{equation*}
    which can be rewritten as 
    \begin{equation*}
        \Pi^w \phi_1 = \left(H^w-z\right)^{-1}\mathscr{O}(|\Im z|^{-1}h^\infty\|\psi\|_{L^2(\R^d,\mathscr{B})}).
    \end{equation*}
    A similar estimate also holds for $\phi_2$:
        \begin{equation*}
        \Pi^w \phi_2 = \left(\Pi^w H^w\Pi^w -z\right)^{-1}\mathscr{O}(|\Im z|^{-1}h^\infty\|\psi\|_{L^2(\R^d,\mathscr{B})}).
    \end{equation*}
    Now, we are able to write
    \begin{equation*}
        \begin{aligned}
            (H^w-z)\phi_1 &= (H^w-z)\Pi^{w,\perp}\phi_1 + (H^w-z)\Pi^w\phi_1\\
            &= (H^w-z)\Pi^{w,\perp}\phi_1 + (H^w-z)\left(H^w-z\right)^{-1}\mathscr{O}(|\Im z|^{-1}h^\infty\|\psi\|_{L^2(\R^d,\mathscr{B})})\\
            &= H^w \Pi^{w,\perp}\Pi^{w,\perp}\phi_1 - z \Pi^{w,\perp}\phi_1 + \mathscr{O}(|\Im z|^{-1}h^\infty\|\psi\|_{L^2(\R^d,\mathscr{B})})\\
            &= (\Pi^{w,\perp}H^w\Pi^{w,\perp}-z)\Pi^{w,\perp}\phi_1 + \mathscr{O}(|\Im z|^{-1}h^\infty\|\psi\|_{L^2(\R^d,\mathscr{B})}),
        \end{aligned}
    \end{equation*}
    where in the last equations we have used the estimates characterizing the superadiabatic projector $\Pi^w$.
    We can rewrite the last equation as
    \begin{equation*}
        (\Pi^{w,\perp}H^w\Pi^{w,\perp}-z)\Pi^{w,\perp}\phi_1 = \Pi^{w,\perp}\psi + \mathscr{O}(|\Im z|^{-1}h^\infty\|\psi\|_{L^2(\R^d,\mathscr{B})}),
    \end{equation*}
    which means that we have obtained the desired estimate:
    \begin{equation*}
        \phi_1 = \phi_2 + \mathscr{O}(|\Im z|^{-2}h^\infty\|\psi\|_{L^2(\R^d,\mathscr{B})}),
    \end{equation*}
    whence~\eqref{eq:egaliteinverse}.
    \smallskip 
    
    We can now terminate the proof.
    By \eqref{eq:HScomposition} and \eqref{eq:egaliteinverse}, we get 
    \begin{equation*}
        \chi(H^w)\Pi^{w,\perp} = \frac{1}{\pi i}\int_{\C} \bar{\partial} \tilde{\chi}(z) \left(\Pi^{w,\perp}H^w\Pi^{w,\perp}-z\right)^{-1}L(\mathrm{d}z) \Pi^{w,\perp} + \mathscr{O}(h^\infty).
    \end{equation*}
    Now, the operator $\Pi^{w,\perp}H^w\Pi^{w,\perp}$ is a selfadjoint operator on $L^2(\R^d,\mathscr{B})$ with domain $L^2(\R^d,\mathscr{A})$. Thus, we can use Helffer-Sjöstrand formula in the right-hand side term and we get
    \begin{equation*}
         \chi(H^w)\Pi^{w,\perp} = \chi(\Pi^{w,\perp}H^w\Pi^{w,\perp})\Pi^{w,\perp}+ \mathscr{O}(h^\infty).
    \end{equation*}
    Note that the principal symbol of the operator $\Pi^{w,\perp}H^w\Pi^{w,\perp}$ is given by $H_0 \Pi_{0}^\perp$. By Equation~\eqref{eq:localspectralH0}, this symbol is bounded  below by $\lambda+\eps$ and using Gärding inequality (see \cite[Théorème 2.1.18]{keraval}), we obtain the following localization of the spectrum of $\Pi^{w,\perp}H^w\Pi^{w,\perp}$:
    \begin{equation*}
        \sigma(\Pi^{w,\perp}H^w\Pi^{w,\perp}) \subset (\lambda+\eps/2,\infty),
    \end{equation*}
    from which we deduce that $\chi(\Pi^{w,\perp}H^w\Pi^{w,\perp}) = 0$ by the support condition on $\chi$. This concludes the proof.
\end{proof}

\section{Applications}\label{sec:application}

\subsection{Electro-magnetic wells in two dimensions}\label{sec.app1}
We prove here Theorem~\ref{thm:app_magnetic_well}.

\subsubsection{Reduction to our framework}
The proof starts as in \cite{Morin_Raymond_VuNgoc}, where the authors perform a transformation of the Hamiltonian through changes of variables and a rescaling, see \cite[Prop. 2.1 \& Lemma 2.4]{Morin_Raymond_VuNgoc}. 
Thanks to these considerations, the operator $h^{-1}\mathscr L_h$ is unitarily equivalent to the pseudodifferential operator 
\[
\mathscr{M}_h={\rm Op}^w_{h,x_2}\left({\rm Op}^w_{1,x_1}(m_h(x,\xi))\right),
\]
where
\begin{multline*}
  m_h(x,\xi)=\dot B(\xi_2+\sqrt{h}x_1,x_2+\sqrt{h}\xi_1)^2 \xi_1^2+ (x_1+\dot \alpha(\xi_2+\sqrt{h}x_1,x_2+\sqrt{h}\xi_1)\xi_1)^2 
\\+ \dot V (\xi_2+\sqrt{h}x_1,x_2+\sqrt{h}\xi_1)+hW(\xi_2+\sqrt{h}x_1,x_2+\sqrt{h}\xi_1)\,,
\end{multline*}
where the functions $\dot B$, $\dot V$ and $\dot \alpha$ are the pull-back by the change of variables $(q_1,q_2)\mapsto (x_1,x_2)$ of the functions $B$, $V$ and $\alpha:=\partial_2A_2(q_1,q_2)$. We also have
$W:=\frac 14 (\partial_1 \dot B)^2 +\frac 14 (\partial_1 \dot\alpha)^2$.

Therefore, we focus on the spectral analysis of $\mathscr{M}_h$ and we consider the spectrum in the window $(\mu_0,\mu_0 + Ch)$ (which is discrete as soon as $h$ is small enough). From \cite[Lemma 2.7]{Morin_Raymond_VuNgoc}, one knows that the corresponding eigenfunctions are microlocalized in $|(x_1,\xi_1)|\leq h^{-\frac12+\delta}$, for $\delta\in(0,\frac12]$. This fact leads to insert cutoff functions and to consider
\begin{equation*}
{\rm Op}^w_{h,x_2}\left({\rm Op}^w_{1,x_1}(p_h(x,\xi))\right),
\end{equation*}
with the symbol $p_h$ explicitly given:
\begin{equation*}
p_h(x,\xi)=\dot B(z_\delta(x,\xi) )^2 \xi_1^2+ (x_1+\dot \alpha(z_\delta(x,\xi))\xi_1)^2 + \dot V (z_\delta(x,\xi))+hW(z_\delta(x,\xi))\,,
\end{equation*}
where 
\[
z_\delta(x,\xi)=( \xi_2+\sqrt h \chi_\delta x_1, x_2+\sqrt h \chi_\delta \xi_1)\,,
\]
with $\chi_\delta= \chi(h^{\frac 12-\delta}(x_1,\xi_1))$ for some cutoff function $\chi\in\mathcal C^\infty(\mathbb R^2,[0,1])$ satisfying $\chi=1$ in a neighborhood of $(0,0)$.
\smallskip 

The operator ${\rm Op}^w_{h,x_2}\left({\rm Op}^w_{1,x_1}(p_h(x,\xi))\right)$ is a semiclassical pseudodifferential operator with  operator-valued symbol in $S(\R^2,\mathscr{L}(B^2(\R),L^2(\R)))$. In fact, we can perform a Taylor expansion of $p_h$ in $X_1 = (x_1,\xi_1)$ near $0$. Given $J\in\mathbb{N}$, we consider the Taylor expansion at order $J$ of $p_h$ and we denote it by $p_h^J$. We have $|p_h(x,\xi)-p_h^J(x,\xi)|\leq Ch^{\frac{J+1}{2}}|\chi_\delta (X_1)|^{J+1}|X_1|^{J+1}$. This leads to consider 
\[
H^{J,w}_h={\rm Op}^w_{h,x_2}{\rm Op}^w_{1,x_1}(p^J_h(x,\xi)),
\]
which essentially satisfies the assumptions of Corollary~\ref{cor:spectrum}, except that the expansion holds here in powers of $\sqrt{h}$. 
Indeed, we can write 
\[
H^J_h = \sum_{n=0}^J h^{\frac n2} H_{\frac n2}\;\;\mbox{with}\;\;H_{\frac n 2} (x_2,\xi_2)= {\rm Op}^w_{1,x_1} (p_n(x,\xi))
\]
where the functions $(x_1,\xi_1)\mapsto p_n(x,\xi)$, once restricted on ${\chi_\delta=1}$, are polynomial functions in $(x_1,\xi_1)$ for all $(x_2,\xi_2)\in \R^2$, with the same parity as $n$. 
Moreover,
\[
p_0 (x,\xi)= \dot B(\xi_2,x_2) \xi_1^2 + (x_1+\dot\alpha(\xi_2,x_2)\xi_1)^2 + \dot V(\xi_2,x_2)\,.
\]
The principal operator symbol $H_0$, equipped with the domain $B^2(\R)$, is a selfadjoint operator and its spectrum is given by the simple eigenvalues 
\begin{equation}\label{def:mu_n}
\mu_p(x_2,\xi_2):=(2p-1) \dot B(\xi_2,x_2) +\dot V(\xi_2,x_2)\,,\quad p\geq 1\,.
\end{equation}
Due to Assumption~\ref{as:app_double_well}, we have $\mu_0=\min\mu_1(x_2,\xi_2)$ and
\[
\forall p\geq 1\,,\forall (x_2,\xi_2)\in\R^2\,,\quad
\mu_{p+1}(x_2,\xi_2) \geq \mu_p(x_2,\xi_2) + 2b_0\,. 
\]
Thus, the assumptions of Corollary \ref{cor:spectrum} are satisfied with $N=1$ and $\lambda=\mu_0+b_0$. Note also that we can write $\Pi_0(x_2,\xi_2)= \left( u_0(x_2,\xi_2),\cdot\right)_{L^2(\mathbb R_{x_1})} u_0(x_2,\xi_2)$ and that the function $x_1\mapsto u_0(x_2,\xi_2,x_1)$ belongs to the Schwartz class and is even for all $(x_2,\xi_2)\in\R^2$.
\smallskip 

The only adaptation to be discussed is the fact that the expansion holds in powers of $\sqrt{h}$, while the quantization procedure stays semiclassical in $h$. 
As observed in Remark~\ref{rem:construct}, 
we still have the existence of a superadiabatic projector corresponding to the eigenvalue $\mu_1(x_2,\xi_2)$. There exists an operator-valued symbol $\Pi_h\sim\Pi_0+\sum_{n\geq 1} h^{\frac n2} \Pi_{\frac n 2}$ satisfying the relations 
\[ 
\Pi_h^w\Pi_h^w=\Pi_h^w +\mathscr{O}(h^\infty)\;\;
\mbox{and}\;\; [H_h^w , \Pi_h^w]= \mathscr{O}(h^\infty).
\]
The same proof as in Theorem~\ref{thm:spectral} allows to construct the symbol $\ell_h=\sum_{n\in\N}h^{\frac n2} \ell_{\frac n2}$ in $S(\R^2_{x_2,\xi_2},\mathscr{L}(L^2(\R_{x_1})))$
	such that 
	\begin{equation*}
		\ell^w (\ell^w)^* = {\rm Id}_{L^2(\R^2,\C)}+\mathscr{O}(h^\infty)\quad \mbox{and}\quad (\ell^w)^* \ell^w = \Pi^w+\mathscr{O}(h^\infty).
	\end{equation*}
	Moreover, $\sigma(H^w)\cap (-\infty,\mu_0+Ch)$ coincides 	up to $\mathscr{O}(h^\infty)$ with  the spectrum of the operator $M^w=\ell^w H^w (\ell^w)^*= \sum_{n\in \N} h^{\frac n2} 
		\lambda_{\frac n2}^w$ with 
\[
\lambda_0(x_2,\xi_2)=\mu_1(x_2,\xi_2)
\;\;\mbox{ and }\;\;
\lambda_{\frac 12}(x_2,\xi_2)=O(h^\infty),\;\;
\forall  (x_2,\xi_2)\in\R^2,
\]
where $\mu_1$ is given in~\eqref{def:mu_n}. The second relation comes from the following observations. By the symbolic calculus and by collecting the terms of order $\sqrt{h}$ in $\ell^w(\ell^*)^w=\mathrm{Id}+\mathscr{O}(h^\infty)$, we have
\[\ell_{\frac 12}\ell^*_0+\ell_0\ell_{\frac12}^*=0\,.
\]
This implies that 
\[
 \ell_{\frac 12}H_0\ell_0^*+\ell_0H_0\ell^*_{\frac 12}=\mu_1(\ell_{\frac 12}\ell^*_0+\ell_0\ell_{\frac12}^*)=0\,.
\]
Finally, the parity properties of $u_0$ and $H_{\frac12}$ yield
\[
\ell_0H_{\frac 12}\ell_0^*=
( H_{\frac 12}(x_2,\xi_2)u_0(x_2,\xi_2), u_0(x_2,\xi_2))_{L^2(\R_{x_1})}=\mathscr{O}(h^\infty)\,,
\]
the last estimate being the price to pay to replace $\chi_\delta$ by $1$ in the expression of $H_{\frac12}$.

\subsubsection{End of the proof of Theorem \ref{thm:app_magnetic_well}}
Consider $\mu\in\mathrm{sp}(\mathscr{M}_h)$ and $\psi$ a corresponding eigenfunction. Due to the microlocalization given in \cite[Lemma 2.7]{Morin_Raymond_VuNgoc}, we have $(H^{J,w}_h-\mu)\psi=\mathscr{O}(h^{\frac{J+1}{2}})$. Conversely, if $\mu\in\mathrm{sp}(H^{J,w}_h)$ and $\psi$ a corresponding eigenfunction, we also have $(\mathscr{M}_h-\mu)\psi=\mathscr{O}(h^{\frac{J+1}{2}})$. With the spectral theorem, this shows that the spectra of $\mathscr{M}_h$ and of $H^{J,w}_h$ in $(\mu_0,\mu_0+Ch)$ coincide modulo $\mathscr{O}(h^{\frac{J+1}{2}})$. Therefore, by using Corollary \ref{cor:spectrum}, we get that the spectra of $\mathscr{M}_h$ and of $M^w$ coincide modulo $\mathscr{O}(h^{\frac{J+1}{2}})$. We are left with the spectral analysis of the operator $M^w$. Under Assumption~\ref{as:app_double_well} (iii), this analysis is a classical result (see \cite{Sjostrand_1992,Charles_Ngoc_2006}). The idea is to relate the operator  $M^w$ to the harmonic oscillator $\mu_0 (-\partial_x^2+x^2)$ via a  Birkhoff normal form, the arguments of which can be transferred to $M^w$, despite its asymptotic expansions in $\sqrt h$, because these terms appear with powers larger than $2$.

\subsection{The Robin magnetic Laplacian} In this section, we prove  Theorem~\ref{theo:Robin}.
It is proved in  \cite[Section 2]{Fahs_LeTreust_Raymond} that the eigenfunctions of $\mathscr{L}_h$ attached to eigenvalues in $[ah,bh]$ are exponentially localized at the scale $h^{\frac12}$ near the boundary of $\Omega$. This suggests to consider a counterclockwise parametrization $\gamma$ by the curvilinear abscissa~$s$ and to perform a local change of coordinates 
\[
\Gamma : (s,t) \mapsto \gamma(s) - t \mathbf{n}(s),\;\; (s,t) \in \mathbb T_{2L} \times (0,t_0),
\]
which gives a parametrization of a neighborhood of the boundary. 

Then, one considers the variable $\tau$ which encodes the distance to the boundary with respect to the scale $\sqrt h$:
\[
t = \sqrt{h} \tau.
\]
These considerations lead to consider the operator (see \cite[Section 3]{Fahs_LeTreust_Raymond}):
\[
\mathfrak N_h \coloneqq \widehat{a}_h(s,\tau)^{-1} \mathcal T_h \widehat{a}_h(s,\tau)^{-1} \mathcal T_h - \widehat{a}_h(s,\tau)^{-1} \partial_\tau \widehat{a}_h(s,\tau)\partial_\tau
\]
acting on functions of $e^{\frac i h \mathfrak f_0 s}L^2\left( \T_{2L} \times \mathbb R_+ \right)$ satisfying the Robin boundary condition $\partial_\tau\psi(s,0)=\gamma\psi(s,0)$. Here,
\[
\widehat{a}_h(s,\tau) = 1 - \sqrt{h} \widehat{\tau} \kappa(s)\;\;\mbox{and}\;\;
\widehat{\tau} = \zeta(h^{\eta}\tau)\tau
\]
with $\eta \in (0,1/2)$, $\zeta$ a smooth cutoff function equal to $1$ near $0$, $\kappa(s)$ being the curvature of the boundary at the point $\Gamma(s)$, and the operator $\mathcal T_h $ is given by 
\[
\mathcal T_h = \Xi^w_0 - \tau + \sqrt{h}\frac{\kappa}{2}\widehat{\tau}^2
\]
where the symbol  $\Xi_0 : \mathbb R \to \mathbb R$ is a smooth increasing function, bounded with all its derivatives. We emphasize that, in this application, the Weyl quantization is considered with the semiclassical parameter $\sqrt{h}$.

As in the previous application, we consider $J$ and write the Taylor expansion of $\mathfrak N_h$, where we replace $\widehat{a}_h(s,\tau)^{-1}$ by its Taylor expansion at the order $J$. We denote by $\mathfrak{N}_h^J$ the corresponding operator. We notice that 
\[\mathfrak{N}_h^0=-\partial_\tau^2+(\Xi_0^w-\tau)^2\,.\]
As in \cite{Fahs_LeTreust_Raymond}, we assume that $\Xi_0$ is also chosen such that, for all $k \in \{1, \dots , N \}$, $\mu_k (\gamma, \Xi_0(\sigma)) = \mu_k (\gamma, \sigma)$ in a neighborhood of $\mu^{-1}_k ([a,b])$ and $\mu_k \circ \Xi_0$ takes its values in $(-\infty, a) \cup (b, +\infty)$ away from it.

Similarly as in the previous application and by using the localization estimates in \cite[Section 2]{Fahs_LeTreust_Raymond}, we see that, if $\mu\in\mathrm{sp}(\mathscr{L}_h)\cap[ah,bh]$ and if $\psi$ is a corresponding eigenfunction, we have
\[(\mathfrak{N}^J_h-\mu)\tilde\psi=\mathscr{O}(h^{\frac{J+1}{2}})\|\tilde\psi\|\,,\]
where $\tilde\psi(s,\tau)=\chi(h^{\frac{\eta}{2}}\tau)\psi\circ\Gamma(s,h^{\frac12}\tau)$.

For more details, see especially \cite[Proposition 3.1]{Fahs_LeTreust_Raymond}. Then, we notice that $\mathfrak{N}^J_h$ (which is not selfadjoint on the canonical $L^2$-space) satisfies the assumptions of Corollary \ref{cor:spectrum}, $N$ being the number of dispersion curves crossing the interval $[a,b]$. From the proof of this corollary, we deduce that 
\[(M^w-\mu)\tilde\Psi=\mathscr{O}(h^{\frac{J+1}{2}})\|\tilde\Psi\|\,.\]
Contrary to the analysis in Section \ref{sec.app1}, we are not in a selfadjoint situation (only the principal symbol of $M^w$ is a priori selfadjoint). Note that the principal symbol of $M^w_h$ is the diagonal of scalar symbols $(\mu_1(\gamma,\Xi_0(\sigma)),\ldots,\mu_N(\gamma,\Xi_0(\sigma)))$. Our assumption that there are no critical points for these scalar symbols in the window $[a,b]$ implies (see \cite[Proposition 1.9]{Fahs_LeTreust_Raymond} and the original article \cite{Rozenblum75}) the existence of a diagonal pseudodifferential operator $A^w$, with symbol in $S(\R^2)$ and real principal symbol, such that $D^w=e^{iA^w}M^we^{-iA^w}$ is a diagonal pseudodifferential operator whose symbol only depends on $\sigma$. The operator $e^{iA^w}$ is unitary modulo $\sqrt{h}$. Thus, we have
\[(D^w-\mu)e^{iA^w}\tilde\Psi=\mathscr{O}(h^{\frac{J+1}{2}})\|e^{iA^w}\tilde\Psi\|\,.\]
Since $D^w$ is a normal operator (it commutes with its adjoint), its resolvent is bounded by the inverse of the distance to the spectrum and we infer that $\mu$ is close to the spectrum of $D^w$ modulo $\mathscr{O}(h^{\frac{J+1}{2}})$.

Conversely, if $\mu\in\mathrm{sp}(D^w)\cap[a,b]$ and if $\Phi$ is a corresponding eigenfunction, we consider $e^{-iA^w}\Phi$, which solves $(M^w-\mu)(e^{-iA^w}\Phi)=0$. Thus, with Corollary \ref{cor:spectrum}, it can be transformed into a quasimode for $\mathfrak{N}^J_h$ (modulo $\mathscr{O}(h^\infty)$) and then, by using the exponential decay in $\tau$, for $\mathfrak{N}_h$ (modulo $\mathscr{O}(h^{\frac{J+1}{2}})$).

\begin{remark}[Exponential decay of $(L^w)^*$ and $(\ell^w)^*$]
According to the exponential decay of the eigenfunction of the De Gennes operator and the boundedness of the function $\Xi_0^w$, the function $u_0$ and all its derivatives have an exponential decay in $\tau$, uniformly in $(s,\sigma)$. Let us consider $\ell,L$ defined in Lemma~\ref{lem:existenceell}, then $L_0 = \ell_0 = (u_0,\cdot)_\mathscr{B}$ and $L^w (\ell^w)^* = {\rm Id}_{L^2(\R^d,\mathscr B)}+\mathscr{O}(h^\infty)$ so $L_0 \ell_1^* + L_1 \ell_0^* + \frac{1}{2i} \left\{L_0, \ell_0^*\right\} = 0.$
Then,
\[
    L_1 = -L_0 \ell_1^*\ell_0 - \frac{1}{2i} \left\{L_0, \ell_0^*\right\}\ell_0\,.
\]
By taking the adjoint, we see that
\[
    L^*_1 = -\ell^*_0 \ell_1 L^*_0 - \frac{1}{2i} \ell^*_0\left\{L_0, \ell_0^*\right\}^*\,.
\]
By using the exponential decay of $u_0$ (and of $\ell_0^*$) and the boundedness of the $\ell_j$ and $L_j$ (and of their derivatives) in $(s,\sigma)$, we get that $L_1^*$ has an exponential decay in $\tau$ uniformly in $(s,\sigma)$. We proceed similarly for $\ell_1^*$ and by induction for the $L_j$ and $\ell_j$ for $j\geq 2$.

\end{remark}

Similar considerations as in \cite[Proposition 3.1]{Fahs_LeTreust_Raymond} allow to get a quasimode for $h^{-1}\mathscr{L}_h$ attached to $\mu$ modulo $\mathscr{O}(h^{\frac{J+1}{2}})$. Since $\mathscr{L}_h$ is selfadjoint, this shows that $\mu h$ is close to the spectrum of $\mathscr{L}_h$ modulo $\mathscr{O}(h^{\frac{J+1}{2}})$.

\appendix

\section{The Sylvester Theorem, proof and comments}\label{sec:sylvester}

We recall that  if $K$ is a selfadjoint operator on a Hilbert space $\mathcal H$, and $\Lambda\subset \sigma(H)$ a subset of the spectrum of $H$ satisfying 
\begin{equation}\label{B0}
\exists \delta>0,\;\; d(\Lambda, \sigma(K)\setminus \Lambda)\geq \delta,
\end{equation}
then, if $\gamma$ is a  closed contour in the complex plane with winding number one around~$\sigma(K)$ (see Figure~\ref{fig:contour}), the spectral projector on the eigenspace associated with $\Lambda$ is given by 
\begin{equation}\label{eq:proj}
\Pi_\Lambda=  -\frac 1{2\pi i} \oint_\gamma (K-z)^{-1} dz. 
\end{equation}
The resolution of the Sylvester problem makes use of this formula and takes advantage of complex analysis.

\begin{theorem}[\cite{bhatia_rosenthal}]
\label{thm:sylvester}
Let $K_0$ and $K_1$ be two selfadjoint operators on some Hilbert space~$\mathcal H$. Assume there exists $\delta>0$ for which $\sigma(K_0)\subset (0, \delta )$ and 
$\sigma(K_1)\subset (\delta,+\infty)$.
Let $\gamma$ be a  closed contour in the plane with winding number one around~$\sigma(K_0)$ and zero around $\sigma(K_1)$.  Then,  for any operator $Y$ the equation 
\begin{equation}\label{sylvester_eq}
K_0X-XK_1=Y
\end{equation}
has a unique solution $X$ that can be expressed as
\begin{equation}\label{solution_sylvester}
	X= -\frac 1{2\pi i} \oint_\gamma (K_0-z)^{-1} Y (K_1-z)^{-1} dz.
\end{equation}
\end{theorem}

\begin{figure}[H]
    \centering
    \includegraphics[width=9cm,height=2cm]{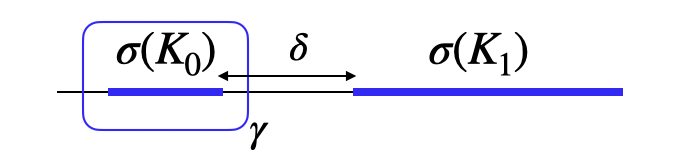}
    \caption{The  contour  $\gamma$
 used in Theorem~\ref{thm:sylvester}}
    \label{fig:contour}
\end{figure}

\begin{proof}
We observe that  for $z\in \gamma$, 
\begin{align*}
&K_0(K_0-z)^{-1} Y (K_1-z)^{-1} - (K_0-z)^{-1} Y (K_1-z)^{-1} K_1\\
&\qquad =
(K_0-z)(K_0-z)^{-1} Y (K_1-z)^{-1}-(K_0-z)^{-1} Y (K_1-z)^{-1}(K_1-z)\\
&\qquad = Y(K_1-z)^{-1} -(K_0-z)^{-1}Y.
\end{align*}
Therefore, we deduce from 
\[ 
\frac 1{2\pi i} \oint_\gamma (K_1-z)^{-1} dz=0 \;\;\mbox{and}\;\;
\frac 1{2\pi i} \oint_\gamma (K_0-z)^{-1} dz = 1,
\]
that the operator $X$ given by~\eqref{solution_sylvester} satisfies~\eqref{sylvester_eq}.
For proving the uniqueness of the solution to~\eqref{sylvester_eq}, we show that $X_0=0$ is the only solution to $K_0X_0=X_0K_1$.
We observe that if $K_0X_0=X_0K_1$, then  for all $n\in \N$, $K_0^nX_0=X_0K_1^n$. 
Since $\sigma(K_1)\subset(\delta,+\infty)$ with $\delta>0$ positive, $K_1$ is invertible with $\|K_1^{-1}\|< \delta ^{-1}$, and
\[
\forall n\in \N, \;\; X_0= (\delta ^{-1} K_0)^n X_0 (\delta  K_1^{-1})^{n}.
\]
Letting $n$ go to $+\infty$, we deduce from $\| \delta^{-1} K_0\|,\|\delta K_1^{-1}\|<1$  that $X_0=0$.
\end{proof}

\medskip

\noindent The explicit formula of the Sylvester problem has important consequences.

\begin{corollary}\label{cor:smooth}
We continue with the setting of  Theorem~\ref{thm:sylvester} and we 
also
assume that $K_0$, $K_1$ and $Y$ depend on some parameters $(x,\xi)$ in $\Omega$,
an open subset of $\R^{2d}$.   
\begin{enumerate}
   \item 
 Assume that $K_0$, $K_1$ and $Y$ are  smooth  (resp. analytic) functions of $(x,\xi)$  in~$\Omega$. Then,  the function $(x,\xi)\mapsto X$ given by~\eqref{solution_sylvester} is smooth (resp. analytic) in $\Omega$. 
\item Assume that $K_0$, $K_1$ and $Y$ are  smooth and that there exists $\nu\in\N$ such that 
\[
\forall \alpha\in\N^{2d},\;\;
\exists C_\alpha>0,\;\;
\forall (x,\xi)\in\Omega,\;\; 
\|\partial^\alpha_z Y(z)\|_{\mathcal L(\mathcal H)}\leq C_\alpha \delta^{-|\alpha|-\nu},
\]
while $K_0$ and $K_1$ have uniformly bounded derivatives.
Then, there exists $C_{\alpha}>0$ such that 
\[
\forall (x,\xi)\in\Omega,\;\; \| \partial_{x,\xi}^\alpha X(x,\xi)\|_{\mathcal H}
\leq C_{\alpha} \delta^{-|\alpha|-\nu-1}
.
\]
\end{enumerate}
\end{corollary}

\begin{proof}
Part (1) comes from the formula~\eqref{solution_sylvester}. For Part (2), we use the formula~\eqref{eq:proj} and choose the contour $\gamma$. We set $a=\inf \sigma(K_0)$, $b=\sup\sigma(K_0)$ and for $\eta>0$, we consider the rectangles $\mathscr R_\eta$ based on the points $(a-\eta,-1-\eta)$, $(a-\eta,1+\eta)$, $(b+\eta,1+\eta)$ and $(b+\eta,-1-\eta)$.
We choose a smooth  curve $\gamma$, strictly included in $(\mathscr R_{\delta} \setminus \mathscr R_0)$ and consisting in four connected branches, $\gamma_\ell$, $1\leq \ell \leq 4$, that are two by two disjoint with 
\[
\gamma_1=\{ z=a- \delta/2-it, \, -1\leq t\leq 1\}, \;\; \gamma_3=\{ z=b+\delta/2 + it, \, -1\leq t\leq 1\}
\]
and $\gamma_2\subset \{1\leq  \Im(z)\leq 1+\delta \}$,  $\gamma_4\subset \{-1-\delta \leq  \Im(z)\leq -1\}$ (see Figure~\ref{fig:contour_rect}).

\begin{figure}[H]
    \centering
    \includegraphics[width=10cm,height=5cm]{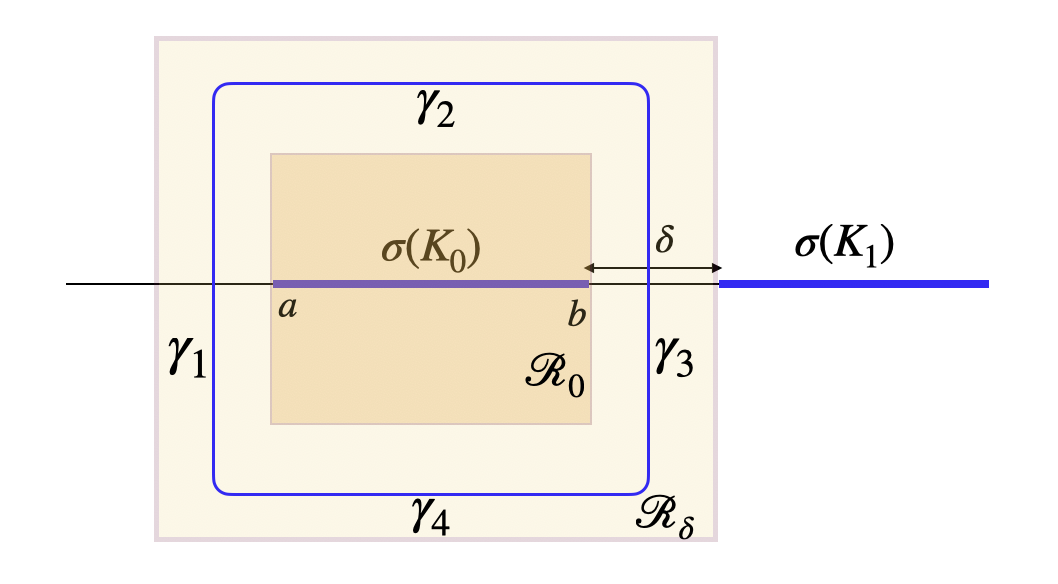}
    \caption{The  contour  $\gamma=\cup_{1\leq i\leq 4}\gamma_i$ used in the proof of Corollary~\ref{cor:smooth}.}
    \label{fig:contour_rect}
\end{figure}

By the standard resolvant estimates and Pythagore's theorem, there exists a constant $c>0$ such that 
\[
\forall z \in\gamma _1\cup \gamma_3,\;\; 
\| (K_0-z)^{-1}\|_{\mathscr L(\mathscr A,\mathscr B)}\leq c\frac 1{\sqrt{t^2+\delta^{2}}},
\]
while for all  $z\in\gamma_2\cup\gamma_4$, this norm is uniformly bounded in $\delta$.
Similarly, 
\[
\forall z \in\gamma_3,\;\; 
\| (K_1-z)^{-1}\|_{\mathscr L(\mathscr A,\mathscr B)}\leq c\frac 1{\sqrt{t^2+\delta^{2}}},
\]
while for all  $z\in\gamma_1\cup \gamma_2\cup\gamma_4$, this norm is uniformly bounded in $\delta$.

 We deduce from the boundedness of the derivatives of $K_0$ and $K_1$ that for all $\alpha\in\N^{2d}$, there exists $c_\alpha>0$ such that 
\begin{align*}
&\forall z \in\gamma _1\cup \gamma_3,\;\; 
\| \partial^\alpha_{x,\xi}\left( (K_0-z)^{-1}\right)\|_{\mathscr L(\mathscr A,\mathscr B)}\leq c_\alpha \frac 1{\left(\sqrt {t^2+\delta^2}\right)^{|\alpha|+1}},\\
&\forall z \in\gamma_3,\;\; 
\| \partial^\alpha_{x,\xi}\left( (K_1-z)^{-1}\right)\|_{\mathscr L(\mathscr A,\mathscr B)}\leq c_\alpha \frac 1{\left(\sqrt {t^2+\delta^2}\right)^{|\alpha|+1}},
\end{align*}
while these derivatives are uniformly bounded, independently of $\delta$ on the other branches of~$\gamma$.

As a consequence, Equation~\eqref{solution_sylvester} implies the  existence of a constant $C_\alpha>0$ associated with  $\alpha\in\N^{2d}$, and   such that 
\begin{align*}
\| \partial^\alpha_{x,\xi} X\|_{\mathscr L(\mathscr A,\mathscr B)}
& \leq C_\alpha \sum_{\alpha_1+\alpha_2+\alpha_3=\alpha} \delta^{-|\alpha_3|-\nu}\int_{-1}^{+1} 
\frac {dt}{\left(\sqrt {t^2+\delta^2}\right)^{|\alpha_1|+|\alpha_2|+2}}\\
& \leq 
C_\alpha \sum_{\alpha_1+\alpha_2+\alpha_3=\alpha} \delta^{-|\alpha|-\nu-1}\int_{-\infty} ^{+\infty} \frac {du}{\left(\sqrt {u^2+1}\right)^{|\alpha_1|+|\alpha_2|+2}}= \mathscr{O}(\delta^{-|\alpha|-\nu-1}),
\end{align*}
which concludes the proof.
\end{proof}

\begin{remark}\label{rem:deriv_proj}
The arguments of the proof of Theorem~\ref{thm:sylvester} also show that by~\eqref{B0}, if $K$ is a smooth function of the parameters $(x,\xi)$, with bounded derivatives  in an open subset $\Omega$ of $\R^d$, then  $\Pi_\Lambda$ given by~\eqref{eq:proj} satisfy 
\[
\forall\alpha\in\N^{2d},\;\;
\| \partial^\alpha \Pi_\Lambda\|_{\mathscr{L}(\mathscr{A},\mathscr{B})} = \mathscr{O}(\delta^{-|\alpha|}).
\]
\end{remark}

\bibliographystyle{abbrv}
\bibliography{biblio}

\end{document}